\documentclass[12pt,a4paper]{article}
\usepackage{graphicx}
\usepackage{amsmath}
\usepackage[cp1251]{inputenc}
\usepackage[T2A]{fontenc}

\usepackage{a4}
\usepackage{amsfonts}
\usepackage{amssymb}

\newtheorem{theorem}{Theorem} [section]

\newtheorem{lemma}[theorem]{Lemma}

\newenvironment{proof}[1][Proof]{\textbf{#1.} }{\ \rule{0.5em}{0.5em}}
\numberwithin{equation}{section}

\begin{document}

\bigskip
%----------------------------------------------------------------------------

\textbf{THE EXISTENCE OF A SMOOTH INTERFACE IN THE EVOLUTIONARY
ELLIPTIC MUSKAT--VERIGIN PROBLEM WITH NONLINEAR SOURCE.}

\bigskip

\textbf{  S.P.Degtyarev}

\bigskip

\textbf{Institute of Applied Mathematics and Mechanics of Ukrainian
National Academy of Sciences,} \textbf{\\74, R. Luxemburg Str.,
Donetsk 83114,  Ukraine }

E-mail: degtyar@i.ua

\bigskip

\textbf{(This is an English translation of the article: \ \ S.P.
Degtyarev, "The existence of a smooth interface in the
Muskat-Verigin elliptic evolution problem with a nonlinear source",
Ukr. Mat. Visn. 7 (2010), no. 3, 301–330,  \ \
http://www.ams.org/mathscinet-getitem?mr=2809065)}

 \begin{flushright} \scriptsize
                           Translated from Russian by V. V. Kukhtin
\end{flushright}

\bigskip

\begin{abstract}
We study the two-phase Muskat--Verigin free-boundary problem for
elliptic equations with nonlinear sources. The existence of a smooth
solution and a smooth free boundary is proved locally in time by
applying the parabolic regularization of a condition on the free
boundary.
\end{abstract}

\bigskip


\bigskip


\bigskip

\bigskip

%================================================================

\section{Statement of the problem and the main result} \label{s1}

The present work is devoted to the study of a boundary-value problem
with unknown boundary describing the process of multidimensional
nonstationary filtration of two fluids in a porous medium
\cite{D1,D2,D3,D4,D5,D6} under assumption that the densities of
these fluids are practically independent of the pressure and, hence,
are constant values. Namely such a situation happens, for example,
in the process of filtration of two unmixed noncompressible fluids
(oil, water) at the displacement of a fluid by a silicate solution,
etc.

A mathematical model of such a problem is the evolutionary problem
with free boundary for elliptic equations.

A three-dimensional problem of such a kind without bulk sources was
considered in \cite{D6}. The goal of the present work is to
generalize results in \cite{D6} to the case of a space with any
dimension and to take the effect of nonlinear bulk sources into
account. As distinct from \cite{D6}, we will study the corresponding
linearized problem within the method of parabolic regularization of
the condition on a free boundary. This method taken by us from
\cite{M} was applied to the given problem in \cite{DN4,DN5}.

We note that the Muskat--Verigin problem для parabolic equations was
studied in \cite{BS} in details and was analyzed in \cite{D17}.
Moreover, with regard for the Gibbs--Thomson condition, the
Muskat--Verigin problem for parabolic equations was considered in
\cite{DN1, DN2, DN3}.

We note also that, in the case of two spatial variables (i.e., at
$N=2$), the problem under study was considered in \cite{DN4,DN5}
with the use of the same regularization as in \cite{M} and in the
present work. Therefore, we will consider that the dimension of the
space $\mathbb R^{N}$ satisfies the condition $N\geq 3$ and use the
method of work \cite{D9}.
%---------------------------------------------------------------

Let $\Omega$ be a doubly connected domain in $\mathbb R^{N}$ with
the boundary $\partial\Omega= \Gamma^{+}\cup \Gamma^{-},$ where
$\Gamma^{\pm}$ are smooth closed surfaces without
self-intersections, $\Gamma(\tau)$ and $\tau\in [0,T]$ are smooth
closed surfaces without self-intersections that lie between
$\Gamma^{\pm}$ and divide the domain $\Omega$ into two doubly
connected domains $\Omega_{\tau}^{\pm}.$ Moreover, $\partial
\Omega_{\tau}^{\pm}=\Gamma(\tau)\cup \Gamma^{\pm},$ and the surface
$\Gamma(0)\equiv \Gamma$ is given.

In the domains $\Omega_{\tau}^{\pm},$ we consider the following
boundary-value problem for the unknown functions $u^{\pm}(y,\tau)$
and the unknown surfaces $\Gamma(\tau)$ with the conditions
\begin{equation}
L_{0}u^{\pm}\equiv \triangle u^{\pm}(y,\tau)=f^{\pm}(u^{\pm}), \quad
y\in \Omega_{\tau}^{\pm}, \label{1.1}
\end{equation}
\begin{equation}
u^{+}|_{\Gamma(\tau)}= u^{-}|_{\Gamma(\tau)}, \label{1.2}
\end{equation}
\begin{equation}
a^{+}(\nabla u^{+},\overrightarrow{N})|_{\Sigma_{T}}=a^{-}(\nabla
u^{-},\overrightarrow{N})|_{\Sigma_{T}}=m\cos
(\overrightarrow{N},\tau), \label{1.3}
\end{equation}
\begin{equation}
u^{\pm}|_{\Gamma_{T}^{\pm}}= g^{\pm}(y,\tau), \label{1.4}
\end{equation}
\begin{equation}
\Gamma(0)=\Sigma_{T}\cap \{t=0\}=\Gamma, \label{1.5}
\end{equation}
\begin{equation}
u^{\pm}(y,0)= u_{0}^{\pm}(y), \quad y\in \Omega^{\pm}. \label{1.5}
\end{equation}
Here, $\Omega^{\pm}$ are the domains into those $\Omega$ is divided
by the initial surface $\Gamma=\Gamma(0),$
$\Gamma_{T}^{\pm}=\Gamma^{\pm}\times [0,T],$ $\Gamma_{T}=\Gamma
\times [0,T],$  $f^{\pm}(u),$ $g^{\pm}(y,\tau),$ and
$u_{0}^{\pm}(y)$ are the given functions, $a^{\pm}$ and $m$ are
positive constants, $\Sigma_{T}=\{(y,\tau):  \tau\in [0,T], \ y\in
\Gamma(\tau) \}$ is a surface in $\mathbb R^{N}\times [0,T],$
$\overrightarrow{N}$ is a normal to $\Sigma_{T}$ in the space
$\mathbb R^{N}\times [0,T]$ that is directed so that its projection
on $\{t= const\}$ is directed toward $\Omega_{\tau}^{+}.$

We note that condition \eqref{1.3} can be presented also in the form
\begin{equation}
a^{+}\frac{\partial u^{+}}{\partial
\overrightarrow{n}(\tau)}=a^{-}\frac{\partial u^{-}}{\partial
\overrightarrow{n}(\tau)} = -mV_{n}, \label{1.6}
\end{equation}
where $V_{n}$ is the velocity of motion of the surface
$\Gamma(\tau)$ along the normal $\overrightarrow{n}(\tau)$ to
$\Gamma(\tau)$ in the space $x\in\mathbb R^{N}$ directed inward
$\Omega_{\tau}^{+}.$ We note that $\overrightarrow{n}(0)$ is a
normal to the initial surface $\Gamma$ denoted below simply by
$\overrightarrow{n}.$

We denote $\Omega_{T}=\Omega\times (0,T)$ and
$\Omega_{T}^{\pm}=\Omega^{\pm}\times (0,T).$ We will use the
standard H\"{o}lder spaces $H^{l}(\overline{\Omega})$ and
$H^{l,l/2}(\overline{\Omega_{T}})$ with a noninteger $l>0$ that were
introduced in \cite{D10} with the norm
\begin{equation*}
 \left|  u\right|
_{\Omega_{T}}^{\left( l\right) }=\sum_{j=0}^{\left[ l\right]
}\sum_{2r+s=j}\left| D_{t}^{r}D_{x}^{s}u\right|  ^{\left(
0\right)  }+\sum_{2r+s=\left[  l\right]  }\left\langle D_{t}^{r}D_{x}%
^{s}u\right\rangle _{x,\Omega_{T}}^{\left(  l-\left[  l\right]
\right) }\\+\sum_{0<l-2r-s<2}\left\langle
D_{t}^{r}D_{x}^{s}u\right\rangle _{t,\Omega_{T}}^{\left(
\frac{l-2r-s}{2}\right)},
\end{equation*}
where  $\left|  D_{t}^{r}D_{x}^{s}u\right|  ^{\left(  0\right)  }%
=\max_{\overline{\Omega_{T}}}\left| D_{t}^{r}D_{x}^{s}u\right|  ,$
$\left\langle u\right\rangle _{x,\Omega_{T}}^{\left( \alpha\right)
}$ and $\left\langle u\right\rangle _{t,\Omega_{T}}^{\left(
\alpha\right) }$ are the H\"{o}lder constants of the function
$u\left( x,t\right) $ in $x$ and in $t,$ respectively.

We also use the following H\"{o}lder spaces.
%-----------------------------------------
Let $\alpha,\beta \in (0,1).$ We define a seminorm (see \cite{D7})
\begin{equation*}
\lbrack u]_{\Omega_{T}}^{(\alpha,\beta)}=\sup_{(x,t),(y,\tau)\in
\overline{\Omega_{T}}}\frac{\left|
u(x,t)-u(y,t)-u(x,\tau)+u(y,\tau)\right| }{\left|  x-y\right|
^{\alpha}\left|  t-\tau\right| ^{\beta}},\quad \alpha,\beta\in\left(
0,1\right) .
\end{equation*}

Then we define the spaces $E^{k+\alpha}(\overline{\Omega_T}),\
k=0,1,2,3,$ with a bounded norm
\begin{equation}
|u|^{(k+\alpha,\alpha)}_{\Omega_{T}}=\max_{t\in
[0,T]}|u(\cdot,t)|^{(k+\alpha)}_{\Omega}+\sum_{|r|=0}^{k}\langle
D^{r}_{x}u\rangle^{(\alpha)}_{t,\Omega_{T}}+
\sum_{|r|=0}^{k}[D^{r}_{x}u]^{(\alpha,\alpha)}_{\Omega_{T}}.
\label{A.1}
\end{equation}
We note that a part of terms in definition \eqref{A.1} can be
interpolated in terms of terms of lower and higher orders.
Therefore,
\begin{equation}
|u|^{(k+\alpha,\alpha)}_{\Omega_{T}}\leq C\bigg(\max_{t\in
[0,T]}|u(\cdot,t)|^{(0)}_{\Omega}+\langle
u\rangle^{(\alpha)}_{t,\Omega_{T}}+
\sum_{|r|=k}[D^{r}_{x}u]^{(\alpha,\alpha)}_{\Omega_{T}}\bigg).
\label{A.2}
\end{equation}
We note also that norm \eqref{A.1} is equivalent to the norm
\begin{equation}
|u|^{(k+\alpha,\alpha)}_{\Omega_{T}}\sim \max_{t\in
[0,T]}|u(\cdot,t)|^{(0)}_{\Omega}+\langle
u\rangle^{(\alpha)}_{t,\Omega_{T}}\\+ \sum_{|r|=k}\sup_{0<h<1}
\left\langle\frac{D^{r}_{x}(x,t+h)-D^{r}_{x}(x,t)}{h^{\alpha}}
\right\rangle ^{(\alpha)}_{x,\Omega_{T}}. \label{A.3}
\end{equation}

%--------------------------------------------------------------

With the use of a local parametrization, we can define smooth
surfaces from the above-indicated classes and the corresponding
spaces of smooth functions on these surfaces in the standard way.

We also define a space $P^{k+\alpha}(\overline{\Gamma_{T}}),$
$k=2,3,$ with the norm
\begin{equation}
|u|_{P^{k+\alpha}(\overline{\Gamma_{T}})}=
|u|_{E^{k+\alpha}(\overline{\Gamma_{T}})}+|u_{t}|_{E^{1+\alpha}(\overline{\Gamma_{T}})}.
\label{A.4}
\end{equation}

%-----------------------------------------

We made the following assumptions about data of problem
\eqref{1.1}--\eqref{1.5}:
\begin{equation}
\Gamma^{\pm}, \Gamma \in H^{4+\alpha}, \quad g^{\pm}\in
H^{4+\alpha,\frac{4+\alpha}{2}}(\Gamma_{T}^{\pm}), \quad
u_{0}^{\pm}(x)\in H^{4+\alpha}(\overline{\Omega^{\pm}}),
\label{1.7a}
\end{equation}
\begin{equation}
 f^{\pm}(u)\in
C^{3}_{loc}(\mathbb R^{1}), \quad (f^{\pm})'(u)\geq \nu >0.
\label{1.7}
\end{equation}
In addition, we suppose that the conditions of consistency hold for
problem \eqref{1.1}--\eqref{1.5} that mean that
\begin{equation}
u_{0}^{\pm}|_{\Gamma^{\pm}}=g^{\pm}(y,0), \quad
u_{0}^{+}|_{\Gamma}=u_{0}^{-}|_{\Gamma}, \quad a^{+}\frac{\partial
u_{0}^{+}}{\partial \overrightarrow{n}}|_{\Gamma}=
a^{-}\frac{\partial u_{0}^{-}}{\partial
\overrightarrow{n}}|_{\Gamma},\label{1.8}
\end{equation}
\begin{equation}
\triangle u_{0}^{\pm}(y)=f^{\pm}(u_{0}^{\pm}(y)),\quad y\in
\overline{\Omega^{\pm}},
 \label{1.8aaa}
\end{equation}
where $\overrightarrow{n}$ is a normal to $\Gamma$ directed toward
$\Omega^{+}.$ We note that condition \eqref{1.8aaa} is, in essence,
the assumption for the right-hand sides $f^{\pm}(u)$ that means that
the corresponding boundary-value problem for Eq. \eqref{1.8aaa} is
solvable. We also assume that the problem is nondegenerate, namely,
\begin{equation}
\frac{\partial u_{0}^{\pm}}{\partial \overrightarrow{n}}\geq \nu
>0, \quad \frac{\partial u_{0}^{+}}{\partial \overrightarrow{n}}-
\frac{\partial u_{0}^{-}}{\partial \overrightarrow{n}}\geq \nu
>0, \quad y\in \Gamma,
\label{1.9}
\end{equation}
where $\nu$ is some positive constant.

In what follows, by $C,$ $b,$ $\nu,$ and $\gamma,$ we denote all
absolute constants or constants depending only on once and for all
fixed data of the problem.

%%============================================

In order to formulate the main result, we introduce a
parametrization of the unknown boundary with the help of some
unknown function \cite{D8} and reduce the initial problem to a
problem in a fixed domain. In so doing, in a sufficiently small
neighborhood $\mathcal{N}$ of the surface $\Gamma,$ we introduce the
coordinates $(\omega,\lambda),$ where $\omega$ are the coordinates
on the surface $\Gamma,$ $\lambda \in R,$ and $|\lambda|\leq
\lambda_{0}$ so that if $x\in \mathcal{N},$ then
\begin{equation}
x=x_{\Gamma}(\omega)+\lambda\overrightarrow{n}(\omega)=x(\omega,\lambda),\quad\left|
\lambda\right|  \leq\lambda_{0} \label{1.8a}%
\end{equation}
in a unique way, where $x_{\Gamma}(\omega)\in \Gamma,$ and $\lambda$
is a deviation of the point $x$ from the surface $\Gamma$ along the
normal $\overrightarrow{n}$ to $\Gamma$ directed, we recall, inward
$\Omega^{+}.$

Let $\rho(\omega,t)$ be a sufficiently small-value function defined
on $\Gamma_{T}=\Gamma \times [0,T],$ $\rho(\omega,0)\equiv 0.$ Then
the parametrization
\[
x=x_{\Gamma}(\omega)+\overrightarrow{n}(\omega)\rho(\omega,t)
\]
at every $t\in [0,T]$ sets some surface $\Gamma_{\rho}(t)$ dividing
the domain $\Omega$ into two subdomains $\Omega_{\rho}^{+}$ and
$\Omega_{\rho}^{-}.$ We denote a surface in $\Omega_{T}\equiv \Omega
\times [0,T]$ by $\Gamma_{\rho,T}\equiv \cup_{t\in
[0,T]}\Gamma_{\rho}(t)\times \{t\}.$ By $\Omega_{\rho,T}^{\pm},$ we
denote those domains, into those the surface $\Gamma_{\rho,T}$
divides the domain $\Omega_{T}.$ We assume (and will prove it below)
that the unknown surface $S_{T}=\Gamma_{\rho,T}$ with some unknown
function $\rho.$

%%================================
Let a function $\chi (\lambda)\in C^{\infty}$ be such that
$\chi(0)=1,$ $0\leq \chi(\lambda)\leq 1,$ $\chi(\lambda) \equiv 0$
at $|\lambda|\geq \lambda_{0},$ $|\chi'|\leq 2/\lambda_{0},$ where
$\lambda_{0}$ is a number from relation \eqref{1.8a}. Let also
$\rho(\omega,t)$ be a function of the class
$S^{2+\alpha}(\Gamma_{T})$ and such that $|\rho|\leq \lambda_{0}/4.$
We define the mapping $(x,t)\rightarrow (y,\tau)$ of the domain
$\Omega_{T}$ onto itself by the formula
\begin{equation}
e_{\rho}:%
\begin{cases}y=\begin{cases}x_{\Gamma}(\omega(x))+\overrightarrow
{n}(\omega(x))(\lambda(x)+\chi(\lambda(x))\rho(\omega(x),t))=x+\overrightarrow{n}%
(\omega(x))\chi(\lambda(x))\rho(\omega,t), & x\in
\mathcal{N},\\
x,& x\notin \mathcal{N},
\end{cases}\\
{\tau=t.}%
\end{cases}
\label{2.1a}%
\end{equation}
Thus, the spatial coordinates $(\omega,\lambda)$ of points
$y=y(x,t)$ and $x$ in a neighborhood $\mathcal{N},$ where the
mapping $e_{\rho}$ differs from the identity one, are connected by
the relations
\[
%\begin{equation}
\omega(y)=\omega(x),\quad\lambda(y)=\lambda(x)+\chi(\lambda(x))\rho
(\omega(x),t). %\label{2.2a}%
%\end{equation}
\]
It is easy to see that the mapping $e_{\rho}$ transfers bijectively
the domains $\overline{\Omega_{T}^{\pm
}}$ onto the domains $\overline{\Omega_{\rho,T}^{%
\pm }}.$ Moreover, since $\rho(\omega,0)\equiv 0,$
$e_{\rho}(x,0)\equiv (x,0).$ For simplicity, we will denote the
functions  $u^{\pm}$ after the change of variables by the same
symbol, i.e.,
\[
u^{ \pm
}(x,t)\equiv u^{%
\pm }(y,\tau)\circ e_{\rho}(x,t),
\]
and $u^{ \pm}(x,t)$ are already defined in the known fixed domains
$\overline{\Omega_{T}^{ \pm }}$ .

%%=================================================

Let us change the variables $(y,\tau)=e_{\rho}(x,t)$ in problem
\eqref{1.1}--\eqref{1.5}. We arrive at the equivalent formulation of
the problem, but already in fixed domains:
\begin{equation}
L_{\rho}u^{\pm}\equiv
\nabla_{\rho}^{2}u^{\pm}(x,t)=f^{\pm}(u^{\pm}), \quad (x,t)\in
\Omega_{T}^{\pm},
 \label{1.12}%
\end{equation}
\begin{equation}
u^{+}(x,t)-u^{-}(x,t)=0, \quad (x,t)\in \Gamma_{T},
 \label{1.13}%
\end{equation}
\begin{equation}
\rho_{t}(\omega,t)+a^{\pm}S(\omega,\rho,\nabla_{\omega}\rho)\frac{\partial
u^{\pm}}{\partial
\overrightarrow{n}}+a^{\pm}\sum_{i=1}^{N}S_{i}(\omega,\rho,\nabla_{\omega}\rho)\frac{\partial
u^{\pm}}{\partial \omega_{i}}=0, \quad (x,t)\in \Gamma_{T},
 \label{1.14}%
\end{equation}
\begin{equation}
u^{\pm}(x,t)=g^{\pm}(x,t), \quad (x,t)\in \Gamma_{T}^{\pm},
 \label{1.15}%
\end{equation}
\begin{equation}
u^{\pm}(x,0)=u_{0}^{\pm}(x), \quad x \in \overline{\Omega^{\pm}},
 \label{1.16}%
\end{equation}
where $\nabla_{\rho}=J_{\rho}\nabla,$ and $J_{\rho}$ is the matrix
inverse and conjugate to the Jacobi matrix $\partial y/\partial x.$
Here, $S(\omega,\rho,\nabla_{\omega}\rho)$ and
$S_{i}(\omega,\rho,\nabla_{\omega}\rho)$ are smooth functions of the
arguments. Moreover,
\begin{equation}
S(\omega,0,0)\equiv 1, \quad \frac{\partial S}{\partial
\rho_{\omega_{i}}}(\omega,0,0)\equiv 0,
 \label{1.17}%
\end{equation}
and
\begin{equation}
S_{i}(\omega,0,0)\equiv 0, \quad i=1,2,\dots,N.
 \label{1.18}%
\end{equation}

We now briefly explain the derivation of relations \eqref{1.14} and
properties \eqref{1.17} and \eqref{1.18}. In the variables
$(y,\tau)$ in a neighborhood $\Gamma_{T},$ let
\begin{equation}
\Phi_{\rho}(y,\tau)\equiv \lambda(y)-\rho(\omega(y),\tau),
 \label{1.19}%
\end{equation}
where $\lambda(y),$ $\omega(y)$ are the $(\omega,\lambda)$
coordinates of a point $y$ in a neighborhood of the surface
$\Gamma.$ We note that the vector $\nabla_{(y,\tau)}\Phi_{\rho}$ is
directed toward the domain $\Omega_{\rho,T}^{+}.$ Thus, the normal
$\overrightarrow{N}$ in condition \eqref{1.3} is
\begin{equation}
\overrightarrow{N}=\nabla_{(y,\tau)}\Phi_{\rho}/|\nabla_{(y,\tau)}\Phi_{\rho}|.
 \label{1.20}%
\end{equation}
Thus, condition \eqref{1.3} can be rewritten in the form
\begin{equation}
a^{\pm}(\nabla_{y}\Phi_{\rho},\nabla_{y}u^{\pm})=m\frac{\partial
\Phi_{\rho}}{\partial \tau}=-m\rho_{\tau}(\omega(y),\tau).
 \label{1.21}%
\end{equation}
In relation \eqref{1.21}, we pass to the variables $(x,t)$ in
correspondence with the change $(x,t)=e_{\rho}(x,t),$ by taking into
account that $\Phi_{\rho}(x,t)\equiv \lambda(x)$ in the variables
$(x,t)$ in a neighborhood of $\Gamma,$ since $\omega(y)=\omega (x)$
and $\lambda(y)=\lambda(x)+\rho(\omega(x),t).$ Thus, relation
\eqref{1.21} takes the form
\begin{equation}
a^{\pm}(J_{\rho}\nabla_{x}\lambda(x),J_{\rho}\nabla_{x}u^{\pm})=-m\rho_{t}(\omega(x),t),
 \label{1.22}%
\end{equation}
or
\begin{equation}
a^{\pm}(J_{\rho}^{*}J_{\rho}\nabla_{x}\lambda(x),\nabla_{x}u^{\pm})+m\rho_{t}(\omega(x),t)=0.
 \label{1.23}%
\end{equation}
Let else $\nabla_{x}=R\nabla_{(\omega,\lambda)},$ where $R$ is a
Jacobi orthogonal matrix of the transition $x\rightarrow
(\omega,\lambda)$ between the coordinates in a neighborhood of
$\Gamma.$ Since
$\nabla_{x}\lambda(x)=\overrightarrow{n}(\omega(x)),$ relation
\eqref{1.23} can be written in the form
\begin{equation}
a^{\pm}(R^{*}J_{\rho}^{*}J_{\rho}R\overrightarrow{n},
\nabla_{(\omega,\lambda)}u^{\pm})+\rho_{t}(\omega,t)=0.
 \label{1.24}%
\end{equation}
In other words, due to the smoothness of the surface $\Gamma$ and,
respectively, the smoothness of the change $x\rightarrow
(\omega,\lambda),$ relation \eqref{1.24} is a relation of the form
\eqref{1.14} with some smooth functions $S$ and $S_{i},$ because
\[
\frac{\partial u^{\pm}}{\partial \lambda}= \frac{\partial
u^{\pm}}{\partial \overrightarrow{n}}.
\]
In this case, at $\rho=0$ and $\rho_{\omega_{i}}=0,$ the matrix
$J_{\rho}\equiv I,$ i.e., relation \eqref{1.24} takes the form
(since $R^{*}R=R^{-1}R=I$ )
\begin{equation}
a^{\pm}\frac{\partial u^{\pm}}{\partial
\overrightarrow{n}}+\rho_{t}(\omega,t)=0,
 \label{1.25}%
\end{equation}
which yields properties \eqref{1.17} and \eqref{1.18}.

Below, we formulate the main result.

\begin{theorem} \label{T1}
Let conditions \eqref{1.7a}--\eqref{1.8aaa} and \eqref{1.9} be
satisfied. Then there exists $T>0$ such that problem
\eqref{1.12}--\eqref{1.16} $($and, hence, problem
\eqref{1.1}--\eqref{1.5}$)$ has the unique smooth solution at $t\in
[0,T],$ and
\begin{equation}
|u^{\pm}|_{E^{2+\alpha}(\overline{\Omega_{T}^{\pm}})}+|\rho|_{P^{2+\alpha}(\Gamma_{T})}\leq
C(u_{0}^{\pm},g^{\pm},\Gamma,\Gamma^{\pm}).
 \label{1.27}%
\end{equation}
\end{theorem}

The subsequent sections are devoted to the proof of Theorem
\ref{T1}.

\section{Linearization of problem \eqref{1.12}--\eqref{1.16}} \label{s2}

To prove Theorem \ref{T1}, we use the method developed in
\cite{D6,D9} that presents, in essence, a version of the Newton
method for the solution of nonlinear equations.

First, we construct the ``initial approximation'' to the solution of
the nonlinear problem \eqref{1.12}--\eqref{1.16}. By
$\sigma(\omega,t),$ we denote a function of the class
$H^{4+\alpha,\frac{4+\alpha}{2}}(\Gamma_{T})$ such that
\begin{equation}
\sigma(\omega,0)=\rho(\omega,0)=0, \quad \frac{\partial
\sigma}{\partial t}(\omega,0)=\rho_{t}(\omega,0)\equiv
\rho^{(1)}(\omega)=a^{\pm}\frac{\partial u_{0}^{\pm}}{\partial
\overrightarrow{n}}.
 \label{2.1}%
\end{equation}

We now continue the functions $u_{0}^{\pm}$ through the surface
$\Gamma$ onto the whole domain $\overline{\Omega}$ and construct the
functions $w^{\pm}(x,t)\in
H^{4+\alpha,\frac{4+\alpha}{2}}(\overline{\Omega_{T}})$ such that
\begin{equation}
w^{\pm}(x,0)=u_{0}^{\pm}(x), \quad x\in \overline{\Omega},
 \label{2.2}%
\end{equation}
\begin{equation}
\frac{\partial w^{\pm}}{\partial t}(x,0)|_{\Gamma}=-\frac{\partial
u_{0}^{\pm}}{\partial
\overrightarrow{n}}(x)|_{\Gamma}\rho_{t}(\omega(x),0)=-\frac{\partial
u_{0}^{\pm}}{\partial
\overrightarrow{n}}(x)|_{\Gamma}\rho^{(1)}(\omega(x)).
 \label{2.3}%
\end{equation}
The procedure of construction of such functions is described in
\cite{D10}. We note that condition \eqref{2.3} implies that the
complicated functions $w^{\pm}\circ e_{\sigma}$ satisfy the relation
\begin{equation}
\frac{\partial (w^{\pm}\circ e_{\sigma})}{\partial
t}(x,0)|_{\Gamma}=\frac{\partial w^{\pm}}{\partial
t}(x,0)+\frac{\partial w^{\pm}}{\partial
\overrightarrow{n}}(x,0)|_{\Gamma}\sigma_{t}(\omega,0)=0.
 \label{2.4}%
\end{equation}
It will be used in what follows and yields, for example, $\partial
F_{3}/\partial t (x,0)=0$ in \eqref{1.32} below. (We note that we
could require that a single condition,
\[
\frac{\partial w^{+}}{\partial t}(x,0)|_{\Gamma}-\frac{\partial
w^{-}}{\partial t}(x,0)|_{\Gamma}=-\Big(\frac{\partial
u_{0}^{+}}{\partial \overrightarrow{n}}(x)|_{\Gamma}-\frac{\partial
u_{0}^{-}}{\partial \overrightarrow{n}}(x)|_{\Gamma}\Big)\rho^{(1)},
\]
be satisfied instead of conditions \eqref{2.3}, which would also
give the necessary result.)

We note that the principal linear part of the mapping $\delta\in
P^{2+\alpha}(\Gamma_{T})\rightarrow w^{\pm}\circ
e_{\sigma+\delta}\in P^{2+\alpha}(\overline{\Omega_{T}})$ is
\[
\lim_{\varepsilon \rightarrow 0} \frac{w^{\pm}\circ
e_{\sigma+\varepsilon\delta}-w^{\pm}\circ
e_{\sigma}}{\varepsilon}=\Big(\frac{\partial w^{\pm}}{\partial
\lambda}\circ
e_{\sigma}\Big)\chi(\lambda(x))\delta(\omega(x),t)\equiv
b^{\pm}(x,t)\delta.
\]
We denote
\begin{equation}
\delta(\omega,t)=\rho(\omega,t)-\sigma(\omega,t), \quad
v^{\pm}(x,t)=u^{\pm}(x,t)-w^{\pm}\circ e_{\sigma}-b^{\pm}\delta.
 \label{1.29}%
\end{equation}

Since any function $f(x,t)$ satisfies the relation
\begin{equation}
(L_{0}f)\circ e_{\rho}=L_{\rho}(f\circ e_{\rho}),
 \label{1.30}%
\end{equation}
we present relations \eqref{1.12}--\eqref{1.16} in the form
\begin{multline}
\triangle v^{ \pm
}(x,t)-(f^{\pm})'(w^{\pm}\circ e_{\sigma})v^{\pm}\\
=\{  (  L_{0}^{%
\pm
}-L_{\sigma}^{%
\pm
})  v^{%
\pm }(x,t) + [f^{\pm}(w^{\pm}\circ
e_{\sigma}+v^{\pm}+b^{\pm}\delta)\\
-(f^{\pm})'(w^{\pm}\circ
e_{\sigma})(v+b^{\pm}\delta)-f^{\pm}(w^{\pm}\circ e_{\sigma})]\\
+ (f^{\pm})'(w^{\pm}\circ e_{\sigma})b^{\pm}\delta +  [
f^{\pm}(w^{\pm}\circ
e_{\sigma})    - (  L_{0}^{%
\pm
}w^{%
\pm })  \circ e_{\sigma+\delta}  ]      \}
\\
+\Big\{  L_{\sigma}^{%
\pm
}\Big(  w^{%
\pm
}\circ e_{\sigma+\delta}-w^{%
\pm
}\circ e_{\sigma}-\Big(  \frac{\partial w^{%
\pm }}{\partial\lambda}\circ e_{\sigma}\Big)
\chi(\lambda)\delta\Big) \\
-(  L_{\sigma+\delta}^{%
\pm
}-L_{\sigma}^{%
\pm
})  \Big(  v^{%
\pm
}+\Big(  \frac{\partial w^{%
\pm }}{\partial\lambda}\circ e_{\sigma}\Big)  \chi\delta\Big)\\ +(
L_{\sigma+\delta}^{%
\pm
}-L_{\sigma}^{%
\pm
})  (w^{%
\pm
}\circ e_{\sigma+\delta}-w^{%
\pm }\circ e_{\sigma})\Big\} \\
\equiv F_{1}^{%
\pm
}(x,t;v^{%
\pm
},\delta)+F_{2}^{%
\pm
}(x,t;v^{%
\pm
},\delta),\quad(x,t)\in\Omega_{T}^{%
\pm
}, \label{1.31}%
\end{multline}%
\begin{equation}
v^{+}-v^{-}+\Big(  \frac{\partial w^{+}}{\partial
\overrightarrow{n}}- \frac{\partial w^{-}}{\partial
\overrightarrow{n}}\Big)\circ e_{\sigma}  \delta=w^{+}\circ
e_{\sigma}-w^{-}\circ e_{\sigma}
\equiv F_{3}(x,t),\quad(x,t)\in\Gamma_{T}, \label{1.32}%
\end{equation}%
\begin{multline}
\delta_{t}+a^{\pm}\frac{\partial v^{\pm}}{\partial
\overrightarrow{n}}+\sum_{i=1}^{N-1}\Big(a^{\pm}\frac{\partial
S_{i}}{\partial
\rho_{\omega_{i}}}(\omega,\sigma,\nabla_{\omega}\sigma)\frac{\partial
w^{\pm}\circ e_{\sigma}}{\partial \omega_{i}}
\Big)\delta_{\omega_{i}}
 = \bigg\{ a^{\pm}[1-
S(\omega,\sigma,\nabla_{\omega}\sigma)]\frac{\partial
v^{\pm}}{\partial \overrightarrow{n}}\\ -\bigg[
\sigma_{t}+a^{\pm}S(\omega,\sigma,\nabla_{\omega}\sigma)\frac{\partial
w^{\pm}\circ e_{\sigma}}{\partial
\overrightarrow{n}}+\sum_{i=1}^{N-1}a^{\pm}S_{i}
(\omega,\sigma,\nabla_{\omega}\sigma)\frac{\partial
w^{\pm}\circ e_{\sigma}}{\partial \omega_{i}} \bigg]\\
-S(\omega,\sigma+\delta,\nabla_{\omega}\sigma+
\nabla_{\omega}\delta)\frac{\partial^{2}w^{\pm}}{\partial
\overrightarrow{n}^{2}}\delta-
\sum_{i=1}^{N-1}a^{\pm}S_{i}(\omega,\sigma,\nabla_{\omega}\sigma)\frac{\partial
v^{\pm}}{\partial \omega_{i}}
-a^{\pm}\sum_{i=1}^{N-1}a^{\pm}S_{i}(\omega,\sigma,\nabla_{\omega}\sigma)\frac{\partial^{2}
w^{\pm}}{\partial \overrightarrow{n} \partial \omega_{i}}\delta
\bigg\}\\ +\bigg\{
-a^{\pm}[S(\omega,\sigma+\delta,\nabla_{\omega}\sigma+\nabla_{\omega}\delta)-
S(\omega,\sigma,\nabla_{\omega}\sigma)] \Big[\frac{\partial
v^{\pm}}{\partial \overrightarrow{n}}+\frac{\partial^{2}
w^{\pm}\circ e_{\sigma}}{\partial \overrightarrow{n}^{2}}\delta
\Big]\\ -a^{\pm}\sum_{i=1}^{N-1}\bigg[
S_{i}(\omega,\sigma+\delta,\nabla_{\omega}\sigma+\nabla_{\omega}\delta)\\-
\sum_{j=1}^{N-1}\frac{\partial S_{i}}{\partial
\rho_{\omega_{j}}}(\omega,\sigma,\nabla_{\omega}\sigma)\delta_{\omega_{j}}
-S_{i}(\omega,\sigma,\nabla_{\omega}\sigma)\bigg] \Big(
\frac{\partial v^{\pm}}{\partial \omega_{i}}+\frac{\partial
w^{\pm}\circ e_{\sigma}}{\partial \omega_{i}}+\frac{\partial^{2}
w^{\pm}\circ e_{\sigma}}{\partial \overrightarrow{n} \partial
\omega_{i}}\delta \Big)\\
 -\sum_{i,j=1}^{N-1}\frac{\partial
S_{i}}{\partial
\rho_{\omega_{j}}}(\omega,\sigma,\nabla_{\omega}\sigma)\delta_{\omega_{j}}\Big(
\frac{\partial v^{\pm}}{\partial \omega_{i}}+\frac{\partial^{2}
w^{\pm}\circ e_{\sigma}}{\partial \overrightarrow{n} \partial
\omega_{i}}\delta \Big) \bigg\}\\ \equiv
F_{4}^{\pm}(x,t,v^{\pm},\delta)+F_{5}^{\pm}(x,t,v^{\pm},\delta),
\label{1.33}
\end{multline}

%-------------------------------------------------------

\begin{equation}
v^{%
\pm
}(x,t)=g^{%
\pm
}(x,t)-w^{%
\pm
}\circ e_{\sigma}\equiv F_{6}^{%
\pm
}(x,t),\quad(x,t)\in\Gamma_{T}^{%
\pm
}, \label{1.34}%
\end{equation}%

\begin{equation}
v^{%
\pm }(x,0)=0,\quad x\in\overline{\Omega^{ \pm}
},\quad\delta(\omega,0)=0, \label{1.35}%
\end{equation}%
where we took into account that $\chi(\lambda)\equiv 1$ in a
neighborhood of $\Gamma,$ and the mappings $e_{\rho}$ and
$e_{\sigma}$ are the identity ones outside of some neighborhood of
$\Gamma.$  Moreover, the procedure of construction of the functions
$w^{\pm}$ and $\sigma$ implies that we search for actually the
functions $v^{\pm}$ and $\delta$ such that, additionally to
\eqref{1.35},
\[
v^{\pm}\in \dot{E}^{2+\alpha}(\overline{\Omega_{T}^{\pm}}), \quad
\delta \in \dot{P}^{2+\alpha}(\Gamma_{T}),
\]
where the dot above the symbol of a space means a subspace
consisting of functions vanishing at $t=0$ together with all
derivatives with respect to $t$ that are admitted by the class.
%===================================
Of basic importance is the circumstance that such classes with a dot
satisfy relations analogous to those for the classes
$\dot{H}^{l,l/2}.$ Namely, if $u,v\in \dot{H}^{l,l/2},$ then
\begin{gather*}
\left|  u\right|
_{H^{l^{\prime},l^{\prime}/2}(\overline{\Omega}_{T})}\leq
CT^{\frac{l-l^{\prime}}{2}}\left|  u\right|  _{H^{l,l/2}(\overline{\Omega}%
_{T})},\\\left|  u\right|
_{S^{l^{\prime}}(\overline{\Omega}_{T})}\leq
CT^{\frac{l-l^{\prime}}{2}}\left|  u\right|  _{S^{l}(\overline{\Omega}_{T}%
)},\\ l^{\prime}<l,
\end{gather*}
\begin{gather*}
\left|  u\right| _{E^{2+\alpha^{\prime}}(\overline{\Omega}_{T})}\leq
CT^{\frac{\alpha-\alpha^{\prime}}{2}}\left|  u\right| _{E^{2+\alpha
}(\overline{\Omega}_{T})}, \\ \left|  u\right|
_{P^{2+\alpha^{\prime}}(\overline{\Omega}_{T})}\leq
CT^{\frac{\alpha-\alpha^{\prime}}{2}}\left|  u\right| _{P^{2+\alpha
}(\overline{\Omega}_{T})}
 \\\alpha^{\prime}<\alpha,
\end{gather*}
\begin{equation}
\left|  uv\right|  _{H^{l,l/2}(\overline{\Omega}_{T})}\leq CT^{\frac{l-[l]}%
{2}}\left|  u\right|  _{H^{l,l/2}(\overline{\Omega}_{T})}\left|
v\right|
_{H^{l,l/2}(\overline{\Omega}_{T})}. \label{2.23}%
\end{equation}
It is easy to see that the right-hand sides $F_{1}-F_{6}$ of
relations \eqref{1.31}--\eqref{1.35} also vanish at $t=0$ and, thus,
belong to the classes with a dot.

%===================================

The sense of relations \eqref{1.31}--\eqref{1.35} consists in the
separation of the principal part in the nonlinear relations
\eqref{1.12}--\eqref{1.16} that is linear in $v^{\pm}$ and $\delta.$
In this case, all ``free terms'' (possessing a high smoothness) and
``quadratic'' terms are transferred to the right-hand side. Then,
using directly the definition of the functions $F_{i}$ on right-hand
sides of \eqref{1.31}--\eqref{1.35} and considering separately each
term, it is easy to verify the validity of the following
proposition. We denote
\[
\mathcal{H}=\dot{E}^{2+\alpha}(\overline{\Omega_{T}^{+}})\times
\dot{E}^{2+\alpha}(\overline{\Omega_{T}^{-}})\times
\dot{P}^{2+\alpha}(\Gamma_{T}), \quad \psi=(v^{+},v^{-},\delta)\in
\mathcal{H},
\]
\begin{equation}
\|\psi\|=|v^{+}|_{E^{2+\alpha}(\overline{\Omega_{T}^{+}})}+
|v^{-}|_{E^{2+\alpha}(\overline{\Omega_{T}^{-}})}+|\delta|_{P^{2+\alpha}(\Gamma_{T})}
 \label{1.36}%
\end{equation}%
and will consider the functions $F_{i}$ as functions of $\psi.$

%=======================================
\begin{lemma} \label{L1}
Let $\psi,$ $\psi_{1},$ $\psi_{2}$ $\in \mathcal{H}.$ Then
\begin{equation}
\left|  F_{1}^{%
\pm
}(x,t;\psi)\right|  _{E^{\alpha}(\overline{\Omega}_{T}^{%
\pm })}\leq C(1+\left\|  \psi\right\|  )T^{\alpha/2}, \label{1.37}
\end{equation}%
\begin{equation}
\left|  F_{1}^{%
\pm
}(x,t;\psi_{2})-F_{1}^{%
\pm
}(x,t;\psi_{1})\right|  _{E^{\alpha}(\overline{\Omega}_{T}^{%
\pm })}C(\left\|  \psi_{i}\right\|  )T^{\alpha/2}\left\|
\psi_{2}-\psi
_{1}\right\|  , \label{1.38}%
\end{equation}%
\begin{equation}
\left|  F_{2}^{%
\pm
}(x,t;\psi)\right|  _{E^{\alpha}(\overline{\Omega}_{T}^{%
\pm })}\leq C(\left\|  \psi\right\|  )\left\|  \psi\right\|  ^{2},
\label{1.39}
\end{equation}%
\begin{equation}
\left|  F_{2}^{%
\pm
}(x,t;\psi_{2})-F_{2}^{%
\pm
}(x,t;\psi_{1})\right|  _{E^{\alpha}(\overline{\Omega}_{T}^{%
\pm })}\leq C(\left\|  \psi_{i}\right\|  )(\left\| \psi_{1}\right\|
+\left\|
\psi_{2}\right\|  )\left\|  \psi_{2}-\psi_{1}\right\|  , \label{1.40}%
\end{equation}%
\begin{equation}
\left|  F_{3}^{%
\pm
}(x,t)\right|  _{E^{2+\alpha}(\Gamma_{T}^{%
\pm
})} \leq CT^{\alpha/2}, \label{1.41}%
\end{equation}%
\begin{equation}
\left|  F_{4}^{%
\pm
}(x,t)\right|  _{E^{1+\alpha}(\Gamma_{T})}\leq CT^{\alpha/2}, \label{1.42}%
\end{equation}%
\begin{equation}
\left|  F_{4}^{%
\pm
}(x,t;\psi_{2})-F_{4}^{%
\pm }(x,t;\psi_{1})\right|  _{E^{1+\alpha}(\Gamma_{T})}\leq
CT^{\alpha/2}\left\|
\psi_{2}-\psi_{1}\right\|, \label{1.43}%
\end{equation}
\begin{equation}
\left|  F_{5}^{%
\pm }(x,t;\psi)\right|  _{E^{1+\alpha}(\Gamma_{T})}\leq C(\left\|
\psi\right\| )\left\|  \psi\right\|  ^{2}, \label{1.44}
\end{equation}%
\begin{equation}
\left|  F_{5}^{%
\pm
}(x,t;\psi_{2})-F_{5}^{%
\pm }(x,t;\psi_{1})\right|  _{E^{1+\alpha}(\Gamma_{T})}\leq
C(\left\| \psi_{i}\right\|  )(\left\|  \psi_{1}\right\| +\left\|
\psi_{2}\right\|  )\left\|  \psi_{2}-\psi_{1}\right\|  , \label{1.45}%
\end{equation}%
\begin{equation}
\left|  F_{6}^{%
\pm
}(x,t)\right|  _{E^{2+\alpha}(\Gamma_{T})}\leq CT^{\alpha/2}, \label{1.46}%
\end{equation}%
where the constants $C(\|\psi_{i}\|)$ remain bounded at bounded
$\|\psi_{i}\|.$
\end{lemma}

%========================================
We note that, while verifying inequality \eqref{1.37}--\eqref{1.46},
it is necessary to consider also relations \eqref{1.17} and
\eqref{1.18}.

%=====================================

%=====================================

\section{Model problem  corresponding to problem \eqref{1.31}--\eqref{1.35}} \label{s3}

In this section, we consider a simple problem corresponding to the
essence of the linear problem that is set by the left-hand sides of
relations \eqref{1.31}--\eqref{1.35}. Such a problem follows from
problem \eqref{1.31}--\eqref{1.35} by fixing the coefficients on the
left-hand sides of \eqref{1.31}--\eqref{1.35} at some point on the
boundary $\Gamma$ at $t=0$ and by a local straightening of the
surface $\Gamma.$ In addition, the boundary conditions corresponding
to conditions on the free boundary are supplemented by the
regularizing term with a small factor $\varepsilon>0$ so as it was
made in \cite{M,DN4,DN5}.

Let $\mathbb R^{N}_{\pm}=\{x\in\mathbb R^{N}:   \pm x_{N}\geq 0\},$
$\mathbb R^{N}_{\pm,T}=\mathbb R^{N}_{\pm}\times [0,T],$
$x'=(x_{1},x_{2},\dots,x_{N-1}).$ We now consider the problem of the
determination of the unknown functions $u^{\pm}(x,t)$ that are set
on $\mathbb R^{N}_{\pm,\infty}=\mathbb R^{N}_{\pm}\times
[0,\infty),$ respectively, and the unknown function $\rho(x',t)$ set
on $\mathbb R^{N-1}_{\infty}=(\mathbb R^{N}\times [0,\infty))\cap
\{x_{N}=0\}$ by the conditions
\begin{equation}
-\triangle u^{\pm}=f_{1}^{\pm}(x,t)=f_{1}^{\pm}(x,t), \quad
(x,t)\in\mathbb R^{N}_{\pm,\infty}, \label{1.47}
\end{equation}
\begin{equation}
u^{+}(x,t)- u^{-}(x,t)+A\rho(x',t)=f_{2}(x',t), \quad x_{N}=0, \
t\geq 0, \label{1.48}
\end{equation}
\begin{equation}
\rho_{t}(x',t)- \varepsilon\triangle_{x'} \rho
+a^{\pm}\frac{\partial u^{\pm}}{\partial
x_{N}}+\sum_{i=1}^{N-1}h_{i}^{\pm}\rho_{x_{i}}=f_{3}^{\pm}(x',t),
\quad x_{N}=0, \ t\geq 0, \label{1.49}
\end{equation}
\begin{equation}
u^{\pm}(x,0)=0, \quad x\in\mathbb R^{N}_{\pm}, \qquad \rho(x',0)=0.
\label{1.50}
\end{equation}
Here, $f_{1}^{\pm},$ $f_{2},$ and $f_{3}^{\pm}$ are finite
functions,
\begin{equation}
f_{1}^{\pm}\in \dot{E}^{\alpha}(\mathbb R^{N}_{\pm,\infty}), \quad
f_{2}\in \dot{E}^{2+\alpha}(\mathbb R^{N-1}_{\infty}), \quad
f_{3}^{\pm}\in \dot{E}^{1+\alpha}(\mathbb R^{N-1}_{\infty}),
\label{1.51}
\end{equation}
$\varepsilon >0$ is a small fixed positive constant, and $A,$
$a^{\pm},$ and $h_{i}^{\pm}$ are given positive constants.

The following proposition is valid.

\begin{theorem} \label{T2}
Let $\varepsilon \in (0,1],$ and let condition \eqref{1.51} be
satisfied. Then, for any finite solution $(u^{+},u^{-},\rho)$ of
problem \eqref{1.47}--\eqref{1.50}, the estimates
\begin{multline}\label{1.52}
|u^{+}|_{E^{2+\alpha}(\mathbb
R^{N}_{+,T})}+|u^{-}|_{E^{2+\alpha}(\mathbb R^{N}_{-,T})}
+|\rho|_{P^{2+\alpha}(\mathbb R^{N-1}_{T})}+  \varepsilon
|\rho|_{P^{3+\alpha}(\mathbb R^{N-1}_{T})}\\
 \leq C\big(
|f^{+}_{1}|_{E^{\alpha}(\mathbb R^{N}_{+,T})}+
|f^{-}_{1}|_{E^{\alpha}(\mathbb R^{N}_{-,T})}
+|f_{2}|_{E^{2+\alpha}(\mathbb
R^{N-1}_{T})}\\+|f^{-}_{3}|_{E^{1+\alpha}(\mathbb R^{N-1}_{T})}+
|f^{+}_{3}|_{E^{1+\alpha}(\mathbb R^{N-1}_{T})} \big),
\end{multline}
are true, and the constant $C$ in \eqref{1.52} is independent of
$\varepsilon.$
\end{theorem}

The subsequent content of this section is devoted to the proof of
Theorem \ref{T2}. In this case, we consider the continuation of all
functions into the region $t<0$ to be zero since all these functions
belong to classes with a dot (i.e., they vanish at $t=0$ together
with their derivatives).

The proof of Theorem \ref{T2} is preceded by the following lemma.
%==============================================
We now consider the boundary-value problem in the half-space
$\mathbb R_{+,T}^{N}$ with a parameter $\varepsilon >0$:
\begin{equation}
-\Delta u=f(z,t), \quad(z,t)\in\mathbb R_{+,T}^{N}, \label{1.54}
\end{equation}
\begin{equation}
\frac{\partial u}{\partial z_{N}} (z^{\prime},0,t)=F(z^{\prime},t),
\quad (z^{\prime},t)\in\mathbb R_{T}^{N-1},
 \label{1.55}%
\end{equation}
\begin{equation}
 u(z,0)=0,\quad  u(z,t)\in \dot{E}^{2+\alpha}(\overline{\mathbb R_{+,T}^{N}}); \label{1.56}%
\end{equation}
where $f\in \dot{E}^{\alpha}(\overline{\mathbb R_{+,T}^{N}}),$ $F\in
\dot{E}^{1+\alpha}(\overline{\mathbb R_{T}^{N-1}}).$ Moreover, $f$
and $F$ are finite in $z,$ or they decrease sufficiently rapidly at
infinity.

\begin{lemma} \label{L3}
Problem \eqref{1.54}--\eqref{1.56} has the unique smooth solution
bounded at infinity that satisfies the estimate
\begin{equation}
\left|  u\right| _{E^{2+\alpha}(\overline{\mathbb R_{+,T}^{N}})}\leq
C(\left| f\right| _{E^{\alpha}(\overline{\mathbb
R_{+,T}^{N}})}+\left| F\right| _{E^{1+\alpha
}(\overline{\mathbb R_{T}^{N-1}})}). \label{1.59}%
\end{equation}%
\end{lemma}

%==============================================

\begin{proof}
We omit a detailed proof of this lemma, since it well known in the
case where the spaces $H^{k+\alpha}$ are used instead of the spaces
$E^{k+\alpha}$ and consists in the well-known estimates of the
potential of a simple layer (see, e.g., \cite{D16} and references
therein). In order to prove the lemma for the spaces $E^{k+\alpha},$
it is sufficient to note that the function
\[
u_{h}(x,t)=\frac{u(x,t)-u(x,t-h)}{h^{\alpha}}, \quad h\in (0,1),
\]
satisfies problem \eqref{1.54}--\eqref{1.56} with the change of the
appropriate functions on the right-hand side of the equation and in
the boundary condition by the functions
\[
f_{h}(x,t)=\frac{f(x,t)-f(x,t-h)}{h^{\alpha}}, \quad
F_{h}(x,t)=\frac{F(x,t)-F(x,t-h)}{h^{\alpha}}.
\]
In this case, by virtue of the assumption about data of the problem,
$f_{h}\in H^{\alpha}(\overline{\mathbb R_{+,T}^{N}}),$ $F_{h}\in
H^{1+\alpha}(\overline{\mathbb R_{T}^{N-1}})$ uniformly in $t$ and
$h.$ Therefore,
\begin{equation*}
\sup_{h}\max_{t}|u_{h}|^{(2+\alpha)}\leq
\sup_{h}\max_{t}C(|f_{h}|^{(\alpha)}+|F_{h}|^{(1+\alpha)})\\
\leq C(\left| f\right| _{E^{\alpha}(\overline{R_{+,T}^{N}})}+\left|
F\right| _{E^{1+\alpha }(\overline{R_{T}^{N-1}})}).
\end{equation*}
This yields, with regard for the definition of the spaces
$E^{k+\alpha},$ the assertion of the lemma.
\end{proof}

%==============================================
An analogous proposition is valid also for problem
\eqref{1.54}--\eqref{1.56} if the Neumann condition \eqref{1.55} is
replaced by the Dirichlet condition
\begin{equation}
u(z^{\prime},0,t)=F(z^{\prime},t), \quad (z^{\prime},t)\in\mathbb
R_{T}^{N-1}.\quad
 \label{1.55aaa}%
\end{equation}

\begin{lemma} \label{L3a}
In \eqref{1.55aaa}, let $F\in \dot{E}^{2+\alpha}(\overline{\mathbb
R_{T}^{N-1}}),$ and let  $F$ be finite in $z$ or sufficiently
rapidly decrease at infinity. Then problem \eqref{1.54},
\eqref{1.55aaa}, \eqref{1.56} has the unique smooth solution bounded
at infinity, for which the following estimate is true:
\begin{equation}
\left|  u\right| _{E^{2+\alpha}(\overline{\mathbb R_{+,T}^{N}})}\leq
C(\left| f\right| _{E^{\alpha}(\overline{\mathbb
R_{+,T}^{N}})}+\left| F\right| _{E^{2+\alpha
}(\overline{\mathbb R_{T}^{N-1}})}). \label{1.59aaa}%
\end{equation}%
\end{lemma}
The proof of this lemma is identical to that of the previous lemma.

%===================================================

Lemmas \ref{L3} and  \ref{L3a} allow us, without any loss of
generality, to consider $f_{1}^{\pm}\equiv 0,$ $f_{2}\equiv 0,$ and
$f^{+}_{3}\equiv 0$ in problem \eqref{1.47}--\eqref{1.50}, so that
only the function $f_{3}^{-}$ is nonzero.

To problem \eqref{1.47}--\eqref{1.50}, we now apply the Laplace
transformation with respect to the variable $t$ and the Fourier
transformation with respect to the variables $x'.$ We denote the
result of such a transformation of the function $f(x',t)$ by
$\widetilde{f}(\xi,p),$ i.e.,
\begin{equation}
\widetilde{f}(\xi,p)=C\int\limits_{0}^{\infty}e^{-pt}\,dt\int\limits_{\mathbb
R^{N-1}}e^{-ix'\xi}f(x,t)\,dx.
 \label{1.60}
\end{equation}
As a result, problem \eqref{1.47}--\eqref{1.50} is reduced to a
boundary-value problem for ordinary differential equations and takes
the form (we denote $h^{\pm}=(h_{1}^{\pm},
h_{2}^{\pm},\dots,h_{N-1}^{\pm})$)
\begin{equation}
\frac{d^{2}\widetilde{u}^{\pm}}{dx_{N}^{2}}-\xi^{2}\widetilde{u}^{\pm}=0,
\quad x_{N}>0 \ (x_{N}<0),
 \label{1.61}
\end{equation}
\begin{equation}
\widetilde{u}^{+}-\widetilde{u}^{-}+A\widetilde{\rho}=0, \quad
x_{N}=0,
 \label{1.62}
\end{equation}
\begin{equation}
\widetilde{\rho}(p+\varepsilon \xi^{2}-ih^{+}\xi)+a^{+}\frac{d
\widetilde{u}^{+}}{dx_{N}}=0, \quad x_{N}=0,
 \label{1.63}
\end{equation}
\begin{equation}
\widetilde{\rho}(p+\varepsilon \xi^{2}-ih^{-}\xi)+a^{-}\frac{d
\widetilde{u}^{-}}{dx_{N}}=\widetilde{f}_{3}^{-} , \quad x_{N}=0.
 \label{1.63a}
\end{equation}
The condition additional to relations \eqref{1.61}--\eqref{1.63a} is
the condition of boundedness at infinity, i.e.,
\begin{equation}
|\widetilde{u}^{\pm}|\leq C, \quad x_{N}\rightarrow \pm \infty.
 \label{1.64}
\end{equation}

With regard for condition \eqref{1.64}, Eq. \eqref{1.61} yields
\begin{equation}
\widetilde{u}^{+}(\xi,p,x_{N})=\widetilde{g}^{+}(\xi,p)e^{-x_{N}|\xi|},
 \label{1.65}
\end{equation}
\begin{equation}
\widetilde{u}^{-}(\xi,p,x_{N})=\widetilde{g}^{-}(\xi,p)e^{x_{N}|\xi|},
 \label{1.66}
\end{equation}
where $\widetilde{g}^{\pm}(\xi,p)=\widetilde{u}^{\pm}|_{x_{N}=0}$
are some functions.

Substituting these formulas in \eqref{1.63} and \eqref{1.63a}, we
obtain
\begin{equation}
\widetilde{\rho}(p+\varepsilon \xi^{2}
-ih^{+}\xi)-a^{+}\widetilde{g}^{+}|\xi|=0, \label{B.0}
\end{equation}
\begin{equation}
\widetilde{\rho}(p+\varepsilon \xi^{2}
-ih^{-}\xi)+a^{-}\widetilde{g}^{-}|\xi|=\widetilde{f}_{3}^{-}.
\label{B.1}
\end{equation}
Let us divide relations \eqref{B.0} and \eqref{B.1} by $a^{+}$ and
$a^{-},$ respectively. Summing these equalities and taking relations
\eqref{1.62} into account, we obtain
\[
\widetilde{\rho}\Big[ \Big(\frac{1}{a^{+}}+\frac{1}{a^{-}} \Big)p
-i\Big(\frac{h^{+}}{a^{+}}+\frac{h^{-}}{a^{-}} \Big)\xi +\varepsilon
\Big(\frac{1}{a^{+}}+\frac{1}{a^{-}} \Big)\xi^{2}+A|\xi|
\Big]=\widetilde{f}_{3}^{-}/a^{-}
\]
or
\begin{equation}
\widetilde{\rho}=\frac{\widetilde{f}}{p-iH\xi+\varepsilon
\xi^{2}+B|\xi|},
 \label{1.69}
\end{equation}
where
\[
\widetilde{f}=\widetilde{f}_{3}^{-}/\Big[a^{-}\Big(\frac{1}{a^{+}}+\frac{1}{a^{-}}
\Big)\Big], \quad H=\Big(\frac{h^{+}}{a^{+}}+\frac{h^{-}}{a^{-}}
\Big)/\Big(\frac{1}{a^{+}}+\frac{1}{a^{-}} \Big),
\]
\[
B=A/ \Big(\frac{1}{a^{+}}+\frac{1}{a^{-}} \Big).
\]
It is obvious that
\begin{equation}
|f|_{E^{1+\alpha}(R^{N-1}_{T})}=C|f^{-}_{3}|_{E^{1+\alpha}(R^{N-1}_{T})}.
 \label{B.2}
\end{equation}

Thus, \eqref{1.69} gives the formula for the unknown function
$\widetilde{\rho}$ in terms of the Fourier--Laplace transform.

We note that the denominator in \eqref{1.69} does not become zero at
$\mathbb{Re}(p)>0.$ By performing the inverse Laplace--Fourier
transformation in \eqref{1.69}, we get
\begin{equation}
\rho(x',t)=\int\limits_{0}^{t}\int\limits_{\mathbb
R^{N-1}}K_{\varepsilon}(x'-\xi,t-\tau)f(\xi,\tau)\,d\xi\, d\tau,
 \label{1.70}
\end{equation}
where $K_{\varepsilon}(x,t)$ is the inverse Laplace--Fourier
transform of the function
$(p-iH\xi+\varepsilon\xi^{2}+B|\xi|)^{-1},$ i.e., ($a>0$)
\begin{equation}
K_{\varepsilon}(x',t)=C\int\limits_{\mathbb
R^{N-1}}e^{ix'\xi}\,d\xi\int\limits_{a-i\infty}^{a+i\infty}\frac{dp}{p-iH\xi+\varepsilon
\xi^{2}+B|\xi|}.
 \label{1.71}
\end{equation}

The inverse Laplace transform of the function
$(p-iH\xi+\varepsilon\xi^{2}+B|\xi|)^{-1}$ can be easily calculated
as
\begin{equation}
\widehat{K_{\varepsilon}}(\xi,t)=e^{iH\xi t-\varepsilon
\xi^{2}t-B|\xi|t},
 \label{1.72}
\end{equation}
where the symbol $\widehat{v}$ stands for the Fourier transform of
the function $v.$ Thus,
\begin{multline}
\rho(x',t)=C\int\limits_{0}^{t}d\tau \int\limits_{\mathbb
R^{N-1}}e^{ix'\xi}e^{iH\xi\tau}e^{-\varepsilon\xi^{2}
\tau}e^{-B|\xi|\tau}\widehat{f}(\xi,t-\tau)\,d\xi\\
=C\int\limits_{0}^{t}d\tau \int\limits_{\mathbb
R^{N-1}}e^{i(x'+H\tau)\xi}e^{-\varepsilon\xi^{2}\tau}e^{-B|\xi|\tau}\widehat{f}(\xi,t-\tau)\,d\xi\\
=\int\limits_{0}^{t}d\tau \int\limits_{\mathbb
R^{N-1}}K_{\varepsilon}(y,\tau)f(x'+H\tau-y,t-\tau)\,dy,
 \label{1.73}
\end{multline}
where
\begin{equation}
K_{\varepsilon}(x',t)=C\int\limits_{\mathbb
R^{N-1}}e^{ix'\xi}e^{-\varepsilon\xi^{2}t}e^{-B|\xi|t}\,d\xi
 \label{1.74}
\end{equation}
is the inverse Fourier transform of the function
$e^{-\varepsilon\xi^{2}t}e^{-B|\xi|t}.$

%=============================================

In view of the well-known properties of the Fourier transformation,
\begin{equation}
K_{\varepsilon}(x',t)=\int\limits_{\mathbb
R^{N-1}}\Gamma_{\varepsilon}(x'-y,t)G(y,t)\,dy=\Gamma_{\varepsilon}*_{x}G,
 \label{A5.1}
\end{equation}
where, as is known,
\begin{equation}
\Gamma_{\varepsilon}(x',t)=C(\varepsilon t
)^{-\frac{N-1}{2}}e^{-\frac{(x')^{2}}{4\varepsilon t}}
 \label{A5.2}
\end{equation}
is the inverse Fourier transform of the function $e^{-\varepsilon
\xi^{2}t}$ (that is the fundamental solution of the heat equation
with the coefficient $\varepsilon$), and $G(x',t)$ is the inverse
Fourier transform of the function $e^{-B|\xi|t}.$

The function $G(x',t)$ can be represented explicitly. To this end,
we note that the solution $V(z)$ of the Dirichlet problem for the
Laplace equation in the half-space $\{z_{N}\geq 0\}\subset\mathbb
R^{N}$ can be expressed in the form (the corresponding Green's
function can be easily constructed, as is known, by the method of
reflection)
\begin{equation}
V(z)=C\int\limits_{\mathbb
R^{N-1}}\frac{z_{N}}{[(z'-\eta)^{2}+z_{N}^{2}]^{\frac{N}{2}}}\varphi(\eta)\,d\eta,
 \label{1.76}
\end{equation}
where $\varphi(z')=V(z',0),$ $z'=(z_{1},\dots,z_{N-1}).$ On the
other hand, such Dirichlet problem in a half-space can be solved
with the use of the Fourier transformation with respect to the
variables $z'$ that gives, as is easily verified,
\begin{equation}
\widehat{V}(\xi,z_{N})=C\widehat{\varphi}(\xi)e^{-z_{N}|\xi|}.
 \label{1.77}
\end{equation}
Comparing \eqref{1.77} and \eqref{1.76}, we may conclude that the
inverse Fourier transform of the function $e^{-z_{N}|\xi|}$ is
$Cz_{N}/(z'^{2}+z_{N}^{2})^{N/2}.$ In other words, by replacing
$z_{N}$ by $Bt,$ we have
\begin{equation}
G(x',t)=C\int\limits_{\mathbb R^{N-1}}e^{ix'\xi}e^{-B|\xi|t}\,d\xi
=C\frac{t}{(x'^{2}+B^{2}t^{2})^{\frac{N}{2}}}, \quad x'\in\mathbb
R^{N-1}.
 \label{1.78}
\end{equation}

%--------------------------------------------------------------

With regard for the well-known properties of the functions
$\Gamma_{\varepsilon}(x',t)$ and the explicitly given $G(x',t)$ in
\eqref{1.78}, we can verify that the kernel $K_{\varepsilon}(x',t)$
possesses the properties (analogous to those of
$\Gamma_{\varepsilon}(x',t)$):
\begin{equation}
\int\limits_{\mathbb
R^{N-1}}D^{r}_{x}D^{s}_{t}K_{\varepsilon}(y,t)\,dy=\begin{cases}1, &
|r|+s=0,\\0, & |r|+s>0.\end{cases}
 \label{A5.3}
\end{equation}
In addition, at any $\varepsilon \in (0,1],$ the derivatives of the
kernel $K_{\varepsilon}(x',t)$ have properties that inherit those of
the kernel $G(x',t)$ in \eqref{1.78}. Namely, the following lemma is
valid.

\begin{lemma} \label{LA5.1}
The function $K_{\varepsilon}(x',t)$ and its derivatives with
respect to $x$ satisfy the estimates
\begin{equation}
|D^{r}_{x}K_{\varepsilon}(x',t)|\leq
C(x'^{2}+t^{2})^{-\frac{N-1}{2}-|r|}, \quad |r|=0,1,2,
 \label{A5.4}
\end{equation}
where the constant $C$ is independent of $\varepsilon.$
\end{lemma}

\begin{proof}
It is easily seen that the function $G(x',t)$ satisfies the
estimates
\begin{equation}
|D^{r}_{x}G(x',t)|\leq Ct((x')^{2}+t^{2})^{-\frac{N}{2}-|r|}\leq
C((x')^{2}+t^{2})^{-\frac{N-1}{2}-|r|}, \quad |r|=0,1,2.
 \label{A5.6}
\end{equation}

For simplicity, we consider only the case $|r|=1,$ since the
remaining cases can be studied quite analogously. Let
$i=\overline{1,N-1}.$ Then
\begin{multline}
K_{\varepsilon x_{i}}(x',t)=\int\limits_{\mathbb
R^{N-1}}\Gamma_{\varepsilon}(y,t)G_{x_{i}}(x'-y,t)\,dy
=\int\limits_{|x'-y|\geq
|x'|/2}\Gamma_{\varepsilon}(y,t)G_{x_{i}}(x'-y,t)\,dy
\\+\int\limits_{|x'-y|<
|x'|/2}\Gamma_{\varepsilon}(y,t)G_{x_{i}}(x'-y,t)\,dy  \equiv
A_{1}+A_{2}.
 \label{A5.5}
\end{multline}
In view of \eqref{1.78}, the estimate
\[
|G_{x_{i}}(x'-y,t)|\leq C(x^{'2}+t^{2})^{-\frac{N}{2}}
\]
is valid for the function $G_{x_{i}}(x'-y,t)$ on the set $|x'-y|\geq
|x'|/2.$ Therefore,
\begin{equation}
|A_{1}|\leq C(x^{'2}+t^{2})^{-\frac{N}{2}}\int\limits_{\mathbb
R^{N-1}}\Gamma_{\varepsilon}(y,t)\,dy=
C(x^{'2}+t^{2})^{-\frac{N}{2}}.
 \label{A5.7}
\end{equation}
Passing to the estimate of $A_{2},$ we note that the quantities
$|y|$ and $|x'|$ are equivalent on the set $|x'-y|<|x'|/2.$
Therefore, on this set with some $\gamma>0,$ we have
\begin{equation}
e^{-\frac{y^{2}}{4\varepsilon t}}\leq
e^{-\gamma\frac{x^{'2}}{\varepsilon t}}.
 \label{A5.8}
\end{equation}
Then we consider two cases. First, let $t\geq |x'|.$ Then, by
estimating $|G_{x_{i}}|\leq Ct^{-N},$ we have

\[
|A_{2}| \leq C (\varepsilon
t)^{-\frac{N-1}{2}}e^{-\gamma\frac{x^{'2}}{\varepsilon t}} \cdot
Measure\{|x'-y|< |x'|/2\}t^{-N}\\=
\]

\begin{equation}
=C\Big(\frac{x'^{2}}{\varepsilon t}
\Big)^{\frac{N-1}{2}}e^{-\gamma\frac{x^{'2}}{\varepsilon
t}}t^{-N}\leq Ct^{-N}\leq C(x^{'2}+t^{2})^{-\frac{N}{2}},
 \label{A5.81}
\end{equation}

since $t\geq |x'|.$

But if $t<|x'|,$ then we apply the integration by parts and
represent $A_{2}$ as follows:
\begin{multline*}
A_{2}=-\int\limits_{|x'-y|<
|x'|/2}\Gamma_{\varepsilon}(y,t)G_{y_{i}}(x'-y,t)\,dy\\
=-\int\limits_{\left|  y-x^{\prime}\right|  =\frac{1}{2}\left|
x^{\prime }\right|
}\Gamma_{\varepsilon}(y,t)G(x^{\prime}-y,t)\,dS_{y} \\
+\int \limits_{\left|  y-x^{\prime}\right|  <\frac{1}{2}\left|
x^{\prime}\right| }\Gamma_{\varepsilon y_{i}}(y,t)G(x^{\prime}-y,t)
 \equiv I_{1}+I_{2}.
\end{multline*}

Taking the relation $|x'-y|=|x'|/2$ into account, by virtue of
\eqref{A5.6} and \eqref{A5.8}, we have
\begin{equation}
\left|  I_{1}\right|  \leq C(\varepsilon
t)^{-\frac{N-1}{2}}e^{-\gamma
\frac{x^{\prime2}}{\varepsilon t}}\left|  x^{\prime}\right|  ^{N-2}%
(x^{\prime2}+t^{2})^{-\frac{N-1}{2}}\\
 \leq C\left|
x^{\prime}\right| ^{-1}(x^{\prime2}+t^{2})^{-\frac{N-1}{2}}\leq
C(x^{\prime2}+t^{2})^{-\frac{N}{2}},\label{A5.9}%
\end{equation}
since $t\leq |x'|.$

Analogously, in view of \eqref{A5.8} and properties of the function
$\Gamma_{\varepsilon y_{i}},$ we have
\[
\left|  I_{2}\right|  \leq C(\varepsilon t)^{-\frac{N}{2}}e^{-\gamma
\frac{x^{\prime 2}}{\varepsilon t}}\int\limits_{\left|
y-x^{\prime}\right| <\frac{1}{2}\left| x^{\prime}\right|
}\frac{t\,dy}{\big[  \left(  x^{\prime }-y\right) ^{2}+t^{2}\big]
^{\frac{N}{2}}}.
\]
Performing the change $x^{\prime}-y=tz$ with $dy=t^{N-1}dz$ in the
last integral, we obtain
\begin{multline}
\left|  I_{2}\right|  \leq C(\varepsilon t)^{-\frac{N}{2}}e^{-\gamma
\frac{x^{\prime2}}{\varepsilon t}}t^{N}t^{-N}\int\limits_{\mathbb
R^{N-1}}\frac
{t\,dy}{(z^{2}+1)^{\frac{N}{2}}}\\
\leq C\Big(  \frac{x^{\prime2}}{\varepsilon t}\Big) ^{N}e^{-\gamma
\frac{x^{\prime2}}{\varepsilon t}}\left| x^{\prime}\right| ^{-N}\leq
C\left|  x^{\prime}\right|  ^{-N}\leq C(x^{\prime2}+t^{2})^{-\frac{N}{2}%
}.\label{A5.10}%
\end{multline}
%------------------------------------------

Thus, estimates \eqref{A5.10}, \eqref{A5.9}, \eqref{A5.81}, and
\eqref{A5.7} yield estimate \eqref{A5.4} of the lemma in the case
$|r|=1.$

The remaining estimates are proved analogously.
\end{proof}

By virtue of properties \eqref{A5.3} and \eqref{A5.4} in the full
analogy to \cite[Chap.~III]{D16}, we obtain the estimate of the
smoothness of potential \eqref{1.73} by the variable $x'$:
\begin{equation}
\max_{t}|\rho(\cdot,t)|^{(2+\alpha)}_{\mathbb R^{N-1}_{T}}\leq C
\max_{t} |f|^{(1+\alpha)}_{\mathbb R^{N-1}_{T}}. \label{A5.11}
\end{equation}

Analogously to the proof of Lemma \ref{L3}, let us consider the
functions
\[
\rho_{h}=\frac{\rho(x',t)-\rho(x',t-h)}{h^{\alpha}}, \quad
u^{\pm}_{h}=\frac{u^{\pm}(x',t)-u^{\pm}(x',t-h)}{h^{\alpha}}.
\]
These functions satisfy the same problem \eqref{1.47}--\eqref{1.50}
with the appropriate right-hand sides. Therefore, we see, by
completely repeating the previous reasoning, that the function
$\rho_{h}(x',t)$ is potential \eqref{1.73} with the density
\[
f_{h}=\frac{f(x',t)-f(x',t-h)}{h^{\alpha}}
\]
instead of $f.$ Thus, \eqref{A5.11} yields
\begin{equation}
|\rho|_{E^{2+\alpha}(\mathbb R^{N-1}_{T})}\leq C
\sup_{t,h}|f_{h}|^{(1+\alpha)_{\mathbb R^{N-1}}}\leq
C|f|_{E^{1+\alpha}(\mathbb R^{N-1}_{T})}.
 \label{A5.12}
\end{equation}

We now represent relation \eqref{1.69} in the form
\[
\rho_{t}-\varepsilon
\triangle_{x'}\rho-B\triangle_{x'}(\Lambda\rho)+\sum_{i=1}^{N}H_{i}\rho_{x_{i}}=f(x',t)
\]
or
\begin{equation}
\rho_{t}-\varepsilon \triangle_{x'}\rho=F(x',t)\equiv
f+B\triangle_{x'}(\Lambda\rho)-\sum_{i=1}^{N}H_{i}\rho_{x_{i}},
 \label{A5.13}
\end{equation}
where $\Lambda \rho$ is the operator with the symbol $|\xi|^{-1},$
i.e., $\Lambda: \ \widetilde{\rho}\rightarrow
\widetilde{\rho}/|\xi|.$ As is well known (see, e.g., \cite{St}),
\[
\Lambda \rho(x',t)=C\int\limits_{\mathbb
R^{N-1}}\frac{\rho(y,t)}{|x'-y|^{N-2}}\,dy.
\]
Moreover, analogously to the standard H\"{o}lder spaces with respect
to $x',$
\begin{equation}
|\Lambda\rho|_{E^{k+\alpha}(\mathbb R^{N-1}_{T})}\leq
C|\rho|_{E^{k-1+\alpha}(\mathbb R^{N-1}_{T})}, \quad k=1,2,3.
 \label{A5.13a}
\end{equation}
Thus, \eqref{A5.13a}, \eqref{A5.13}, and \eqref{A5.12} yield
\begin{equation}
|F|_{E^{1+\alpha}(\mathbb R^{N-1}_{T})}\leq
C(|\rho|_{E^{2+\alpha}(\mathbb R^{N-1}_{T})}+
|f|_{E^{1+\alpha}(\mathbb R^{N-1}_{T})})\leq
C|f|_{E^{1+\alpha}(\mathbb R^{N-1}_{T})}.
 \label{A5.14}
\end{equation}
%-----------------------------------

Since $f(x^{\prime},0)=0$ and
$\rho(x^{\prime},0)=\rho_{t}(x^{\prime},0)=0,$ Eq. \eqref{A5.13}
yields, completely analogously to \cite[Chap.~IV]{D10},
\begin{equation}
\varepsilon\left\langle D_{x}^{3}\rho\right\rangle _{x,\mathbb
R_{T}^{N-1}}^{(\alpha )}\leq C\left\langle D_{x}F\right\rangle
_{x,\mathbb R_{T}^{N-1}}^{(\alpha
)}.\label{A5.15}%
\end{equation}
In addition, with regard for the function
 $\rho_{h}=(\rho(x^{\prime},t)-\rho(x^{\prime},t-h))/h^{\alpha},$ we have
\begin{equation}
\varepsilon\left[  D_{x}^{3}\rho\right] _{x,t,\mathbb
R_{T}^{N-1}}^{(\alpha)}\leq C\left[  D_{x}F\right] _{x,t,\mathbb
R_{T}^{N-1}}^{(\alpha)}.
\label{A5.16}%
\end{equation}
Then it follows from \eqref{A5.13} that
\begin{equation}
\left\langle D_{x}\rho_{t}\right\rangle _{x,\mathbb
R_{T}^{N-1}}^{(\alpha)}+\left[ D_{x}\rho_{t}\right] _{x,t,\mathbb
R_{T}^{N-1}}^{(\alpha)}\leq C\big(  \left\langle D_{x}F\right\rangle
_{x,\mathbb R_{T}^{N-1}}^{(\alpha)}+\left[ D_{x}F\right]
_{x,t,\mathbb R_{T}^{N-1}}^{(\alpha)}\big).
\label{A5.17}%
\end{equation}

By virtue of the finiteness of the function $\rho(x',t)$ relations
\eqref{A5.15}--\eqref{A5.17} and \eqref{A5.14} yield
\begin{equation}
\left|  \rho\right| _{E^{2+\alpha}(\mathbb R_{T}^{N-1}
)}+\varepsilon\left| \rho\right| _{E^{3+\alpha}(\mathbb R_{T}^{N-1}
)}+\left|  \rho_{t}\right| _{E^{1+\alpha}(\mathbb R_{T}^{N-1} )}\leq
C\left|  f\right|  _{E^{1+\alpha}(\mathbb R_{T} ^{N-1})},
\label{A5.18}%
\end{equation}
where the constant $C$ is independent of $\varepsilon,$ i.e.,
\begin{equation}
\left|  \rho\right| _{P^{2+\alpha}(\mathbb R_{T}^{N-1}
)}+\varepsilon\left| \rho\right| _{P^{3+\alpha}(\mathbb R_{T}^{N-1}
)}\leq C\left|  f\right| _{E^{1+\alpha}(\mathbb R_{T}^{N-1} )},
\label{A5.19}%
\end{equation}
which gives the required estimate for the function $\rho(x',t).$

Possessing the estimate for the function $\rho(x',t),$ we can
consider the functions $u^{\pm}(x,t)$ as solutions of the Neumann
problems in the appropriate domains with the condition
\begin{equation}
a^{\pm}\frac{\partial u^{\pm}}{\partial
x_{N}}\Big|_{x_{N}=0}=F^{\pm}_{1}=f^{\pm}_{3}-\rho_{t}+\varepsilon
\Delta_{x^{\prime}}\rho-h^{\pm}\nabla\rho.
\label{A5.20}%
\end{equation}
In this case, $F^{\pm}_{1}$ are finite, and, by virtue of
\eqref{A5.19}, we have
\begin{equation}
\left|  F^{\pm}_{1}\right|  _{E^{1+\alpha}(\mathbb R_{T}^{N-1}
)}\leq C\left| f\right| _{E^{1+\alpha}(\mathbb R_{T}^{N-1} )}.
\label{A5.21}%
\end{equation}
Then Lemma \ref{L3} yields
\begin{equation}
\left|  u^{\pm}\right|  _{E^{1+\alpha}(\mathbb R_{T}^{N-1} )}\leq
C\left| f\right| _{E^{1+\alpha}(\mathbb R_{T}^{N-1} )}.
\label{A5.22}%
\end{equation}

Estimate \eqref{A5.22} together with estimate \eqref{A5.19} complete
the proof of Theorem \ref{T2}. $\square$

%-------------------------------------

\section{A linear problem in the domain $\Omega_{T}$} \label{s4}

In this section, we prove the solvability of a linear problem
corresponding to problem \eqref{1.31}--\eqref{1.35} with the given
right-hand sides from the appropriate classes. In this case, we
regularize the boundary condition on the surface $\Gamma$ in the
same manner, as it was made in \cite{M,DN4,DN5}.

We now consider the problem of the determination of the unknown
functions $u^{\pm}(x,t)$ defined in the domains
$\overline{\Omega}^{\pm}_{T},$ respectively, and the unknown
function $\rho(x,t)$ defined on the surface $\Gamma_{T}$ by the
conditions
\begin{equation}
-\triangle u^{\pm}+b^{\pm}(x,t)u^{\pm}=f_{1}^{\pm}, \quad (x,t)\in
\Omega_{T}^{\pm}, \label{1.91}
\end{equation}
\begin{equation}
u^{+}-u^{-}+A(x,t)\rho=f_{2}, \quad (x,t)\in \Gamma_{T},
\label{1.92}
\end{equation}
\begin{equation}
\rho_{t}-\varepsilon \triangle_{\Gamma}\rho+a^{\pm}\frac{\partial
u^{\pm}}{\partial
\overrightarrow{n}}+\sum_{i=1}^{N-1}h_{i}^{\pm}(x,t)\rho_{\omega_{i}}=f_{3}^{\pm},
 \quad (x,t)\in \Gamma_{T},
\label{1.93}
\end{equation}
\begin{equation}
u^{\pm}(x,0)=0, \quad \rho(x,0)=0, \label{1.94}
\end{equation}
\begin{equation}
u^{\pm}(x,t)=f_{4}^{\pm},  \quad (x,t)\in \Gamma_{T}^{\pm},
\label{1.95}
\end{equation}
where $\triangle_{\Gamma}$ is the Laplace operator on the surface
$\Gamma$ (see, e.g., \cite{M}), $a^{\pm}=const>0,$ $b^{\pm}(x,t)\in
E^{\alpha}(\overline{\Omega}_{T}^{\pm}),$ $A(x,t)\in
E^{2+\alpha}(\Gamma_{T}),$ $h_{i}^{\pm}(x,t)\in
E^{1+\alpha}(\Gamma_{T}),$ and the conditions $\nu\leq A(x,t)\leq C$
and $b^{\pm}(x,t)\geq \nu>0$ are satisfied. The right-hand sides
$f_{i}$ in relations \eqref{1.91}--\eqref{1.95} are assumed to be
such that the quantities
\[
\mathcal{M}^{\pm}_{T}\equiv
|f_{1}^{\pm}|_{E^{\alpha}(\overline{\Omega}^{\pm}_{T})}+
|f_{2}|_{E^{2+\alpha}(\Gamma_{T})}+|f_{3}^{\pm}|_{E^{1+\alpha}(\Gamma_{T})}+
|f_{4}^{\pm}|_{E^{2+\alpha}(\Gamma^{\pm}_{T})}<\infty,
\]
\begin{equation}
 \mathcal{M}_{T}\equiv \mathcal{M}^{+}_{T}+\mathcal{M}^{-}_{T}\label{1.96}
\end{equation}
are finite, and
\begin{equation}
f^{\pm}_{1}(x,0)=0 , \quad f_{2}(x,0)=0, \quad f_{3}^{\pm}(x,0)=0,
\quad f_{4}^{\pm}(x,0)=0, \label{1.97}
\end{equation}
i.e., all functions $f_{k}$ belong to spaces marked with a dot.

The following theorem is valid.

\begin{theorem} \label{T3}
If conditions \eqref{1.96} and \eqref{1.97} are satisfied, problem
\eqref{1.91}--\eqref{1.95} has the unique solution at any
$\varepsilon > 0$ from the space $u^{\pm}\in
\dot{E}^{2+\alpha}(\overline{\Omega}_{T}^{\pm}),$ $\rho \in
\dot{P}^{3+\alpha}(\Gamma_{T}),$ and the estimate
\begin{equation}
|u^{+}|_{E^{2+\alpha}(\overline{\Omega}^{+}_{T})}+
|u^{-}|_{E^{2+\alpha}(\overline{\Omega}^{-}_{T})}+
|\rho|_{P^{2+\alpha}(\Gamma_{T})}+\varepsilon
|\rho|_{P^{3+\alpha}(\Gamma_{T})}\leq C_{T}\mathcal{M}_{T},
 \label{1.99}
\end{equation}
where the constant $C_{T}$ from \eqref{1.99} is independent of
$\varepsilon\in (0,1],$ is true.

At $\varepsilon=0,$ problem \eqref{1.91}--\eqref{1.95} has the
unique solution from the spaces $u^{\pm}\in
\dot{E}^{2+\alpha}(\overline{\Omega}_{T}^{\pm}),$ $\rho \in
\dot{P}^{2+\alpha}(\Gamma_{T}),$ and estimate \eqref{1.99} with
$\varepsilon=0$ is valid.
\end{theorem}
%-----------------------------------------------------------

\begin{proof}
First, we prove estimate \eqref{1.99}, by assuming the availability
of a solution of problem \eqref{1.91}--\eqref{1.95} from the
appropriate class. The following lemma is valid.

\begin{lemma} \label{L4.1}
For any solution of problem \eqref{1.91}--\eqref{1.95} from the
class $u^{\pm}\in \dot{E}^{2+\alpha}(\overline{\Omega}_{T}^{\pm}),$
$\rho \in \dot{P}^{3+\alpha}(\Gamma_{T}),$ estimate \eqref{1.99} is
valid at $\varepsilon>0.$
\end{lemma}

\begin{proof}
By using the standard Schauder technique for estimates and by
considering the results of Section \ref{s3} on properties of the
model problems corresponding to points of the boundary $\Gamma,$ we
obtain, in the ordinary manner, the following \textit{a priori}
estimate of the solution of problem \eqref{1.91}--\eqref{1.95}:
\begin{equation}
|  u^{+}| _{E^{2+\alpha}(\overline{\Omega}_{T}^{+})}+| u^{-}|
_{E^{2+\alpha}(\overline{\Omega}_{T}^{-})}+\left| \rho\right|
_{P^{2+\alpha}(\Gamma_{T})}+\varepsilon\left| \rho\right|
_{P^{2+\alpha }(\Gamma_{T})}\\
\leq C\mathcal{M}_{T}+C\big(  \left\langle u^{+}\right\rangle _{t,\Omega_{T}^{+}%
}^{(\alpha)}+\left\langle u^{-}\right\rangle
_{t,\Omega_{T}^{-}}^{(\alpha
)}\big).  \label{A6.1}%
\end{equation}
Whereas $|u^{\pm}|^{(0)}_{\Omega_{T}^{\pm}}\leq
CT^{\alpha}\left\langle u^{\pm}\right\rangle
_{t,\Omega_{T}^{\pm}}^{(\alpha)},$ the H\"{o}lder constants
$\left\langle u^{+}\right\rangle _{t,\Omega_{T}^{+}}^{(\alpha)} $
and $\left\langle u^{-}\right\rangle _{t,\Omega_{T}^{-}}^{(\alpha)}$
for the functions $u^{\pm}(x,t)$ with respect to the variable $t$
cannot be estimated by the interpolation in the space
$\dot{E}^{2+\alpha}(\overline{\Omega}_{T}^{\pm}).$ In order to
estimate these H\"{o}lder constants, we consider the functions
($h\in (0,1)$)
\[
u_{h}^{\pm}=\frac{u^{\pm}(x,t)-u^{\pm}(x,t-h)}{h^{\alpha}}.
\]
and estimate their modulus maximum $\left|  u_{h}^{\pm}\right|
_{\overline{\Omega}_{T}^{\pm}}^{(0)}$ uniformly in $h.$ We take into
account that $u_{h}^{\pm}\in H^{2+\alpha}(\overline{\Omega}^{\pm}),$
$t\in\lbrack0,T],$ and the space
$H^{2+\alpha}(\overline{\Omega}^{\pm})$ is compactly embedded in the
space $L_{\infty} (\overline{\Omega}^{\pm}) \subset L_{2}
(\overline{\Omega}^{\pm}).$ Therefore, for any $\delta>0$ and $t\in
[0,T],$ we have the inequality (see \cite{Lions})
\begin{equation}
\left|  u_{h}^{\pm}\right|
_{\overline{\Omega}^{\pm}}^{(0)}\leq\delta\left| u_{h}^{\pm}\right|
_{\overline{\Omega}^{\pm}}^{(2+\alpha)}+C_{\delta}\left\|
u_{h}^{\pm}\right\|  _{2,\overline{\Omega}^{\pm}},\label{A6.2}%
\end{equation}
where $\left\|  u_{h}^{\pm}\right\| _{2,\overline{\Omega}^{\pm}}$ is
the $L_{2}$-norm of the functions $u_{h}^{\pm}.$ Relation
\eqref{A6.2} yields
\begin{equation}
\left|  u_{h}^{\pm}\right| _{\overline{\Omega}_{T}^{\pm}}^{(0)}\leq
\delta\left|  u^{\pm}\right|  _{E^{2+\alpha}(\overline{\Omega}_{T}^{\pm}%
)}+C_{\delta}\max_{t\in\lbrack0,T]}\left\|  u_{h}^{\pm}\right\|
_{2,\overline
{\Omega}^{\pm}}.\label{A6.3}%
\end{equation}
Thus, we need to estimate the quantity $\max_{t\in[0,T]}\left\|
u_{h}^{\pm}\right\| _{2,\overline{\Omega}^{\pm}}.$

Without any loss of generality, we can consider that
$f^{\pm}_{4}\equiv 0,$ since these functions can be extended inward
$\overline{\Omega}^{\pm}_{T}$ with the preservation of a class. Then
we can consider new unknown functions $v^{\pm}=u^{\pm}-f^{\pm}_{4}$
that satisfy the same problem with the same estimate of the
right-hand sides. Thus, by taking $f^{\pm}_{4}\equiv 0$ without any
loss of generality and by subtracting relation \eqref{1.93} for the
sign ``$-$'' from the same relation for the sign ``$+$'', we get a
problem for the functions $u^{\pm}_{h}$:
\begin{equation}
-\bigtriangleup u_{h}^{\pm}+b^{\pm}u_{h}^{\pm}=F_{1}^{\pm}\equiv f_{1h}^{\pm}%
-b_{h}^{\pm}\widetilde{u}^{\pm},\quad(x,t)\in\Omega_{T}^{\pm},\label{A6.4}%
\end{equation}%
\begin{equation}
u_{h}^{+}-u_{h}^{-}=F_{2}\equiv-A\rho_{h}-A_{h}\widetilde{\rho}+f_{2h}%
,\quad(x,t)\in\Gamma_{T},\label{A6.5}%
\end{equation}%
\begin{equation}
a^{+}\frac{\partial
u_{h}^{+}}{\partial\overrightarrow{n}}-a^{-}\frac{\partial
u_{h}^{-}}{\partial\overrightarrow{n}}=F_{3} \equiv
f_{3h}^{+}-f_{3h}^{-}+\overrightarrow{H}\nabla_{\omega}\rho
_{h}+\overrightarrow{H}_{h}\nabla_{\omega}\widetilde{\rho},\quad
(x,t)\in
\Gamma_{T},\label{A6.6}%
\end{equation}%
\begin{equation}
u_{h}^{\pm}=0,\quad(x,t)\in\Gamma_{T}^{\pm},\label{A6.7}%
\end{equation}%
\begin{equation}
u_{h}^{\pm}(x,0)=0,\quad x\in\overline{\Omega}^{\pm},\label{A6.8}%
\end{equation}
where the lower index $h$ of the designation of functions means the
corresponding difference relation,
$\widetilde{u}^{\pm}(x,t)=u^{\pm}(x,t-h),$   $\widetilde{\rho
}(x,t)=\rho(x,t-h),$ $\overrightarrow{H}=\{h_{i}^{+}-h_{i}^{-}\}.$

Let us multiply Eqs. \eqref{A6.4} by the functions
$a^{\pm}u_{h}^{\pm},$ respectively, and integrate by parts over the
domains $\Omega^{\pm}.$ With regard for the direction of the normal
$\overrightarrow{n}$ to the boundary $\Gamma,$ we obtain
\begin{equation}
a^{\pm}\int\limits_{\Omega^{\pm}}\left(  \nabla u_{h}^{\pm}\right)  ^{2}%
dx+a^{\pm}\int\limits_{\Omega^{\pm}}b^{\pm}\left(
u_{h}^{\pm}\right) ^{2}dx \\
\pm \int\limits_{\Gamma}u_{h}^{\pm}\Big(  a^{\pm}\frac{\partial u_{h}^{\pm}%
}{\partial\overrightarrow{n}}\Big)\,  dS=a^{\pm}\int\limits_{\Omega^{\pm}}%
u_{h}^{\pm}F_{1}^{\pm}\,dx.\label{A6.9}%
\end{equation}
Since the relation $u_{h}^{+}=u_{h}^{-}+F_{2}$ is satisfied on the
surface $\Gamma$\, relation \eqref{A6.9} for the sign $'+'$ can be
presented in the form
\begin{multline}
a^{+}\int\limits_{\Omega^{+}}\left(  \nabla u_{h}^{+}\right)  ^{2}dx+a^{+}%
\int\limits_{\Omega^{+}}b^{+}\left(  u_{h}^{+}\right)  ^{2}dx+\int
\limits_{\Gamma}u_{h}^{-}\Big(  a^{+}\frac{\partial
u_{h}^{+}}{\partial \overrightarrow{n}}\Big)\,  dS\\
=a^{+}\int\limits_{\Omega^{+}}u_{h}^{+}F_{1}^{+}dx-\int\limits_{\Gamma}%
F_{2}\Big(  a^{+}\frac{\partial u_{h}^{+}}{\partial\overrightarrow{n}%
}\Big)\,  dS.\label{A6.10}%
\end{multline}
Adding \eqref{A6.10} and relation \eqref{A6.9} for the sign ``$-$''
and taking conditions \eqref{A6.6} into account, we get
\begin{multline}
a^{+}\int\limits_{\Omega^{+}}\left(  \nabla u_{h}^{+}\right)  ^{2}dx+a^{-}%
\int\limits_{\Omega^{-}}\left(  \nabla u_{h}^{-}\right)  ^{2}dx+a^{+}%
\int\limits_{\Omega^{+}}b^{+}\left(  u_{h}^{+}\right)  ^{2}dx+a^{-}%
\int\limits_{\Omega^{-}}b^{-}\left(  u_{h}^{+}\right)  ^{2}dx\\
=a^{+}\int\limits_{\Omega^{+}}u_{h}^{+}F_{1}^{+}dx+a^{-}\int\limits_{\Omega
^{-}}u_{h}^{-}F_{1}^{-}dx-\int\limits_{\Gamma}F_{2}\Big(
a^{+}\frac{\partial
u_{h}^{+}}{\partial\overrightarrow{n}}\Big)\,  dS-\int\limits_{\Gamma}%
u_{h}^{-}F_{3}\,dS.\label{A6.11}%
\end{multline}
We now estimate the terms on the right-hand side of \eqref{A6.11},
by using the Cauchy inequality with small parameter $\mu>0$:
\begin{equation}
\Bigg|\, \int\limits_{\Omega^{\pm}}u_{h}^{\pm}F_{1}^{\pm}\,dx\Bigg|
\leq\mu ^{2}\left\|  u_{h}^{\pm}\right\|
_{2,\Omega^{\pm}}^{2}+C_{\mu}\left\| F_{1}^{\pm}\right\|
_{2,\Omega^{\pm}}^{2}\\
\leq C\mu^{2}\left|  u^{\pm}\right|^2  _{E^{2+\alpha}(\overline{\Omega}_{T}%
^{\pm})}+C_{\mu}\big(  \left|  F_{1}^{\pm}\right|  _{\Omega_{T}^{\pm}}%
^{(0)}\big)  ^{2},\label{A6.12}%
\end{equation}%
\begin{equation}
\Bigg|\,  \int\limits_{\Gamma}F_{2}\Big(  a^{+}\frac{\partial u_{h}^{+}%
}{\partial\overrightarrow{n}}\Big) \, dS\Bigg| \leq\mu^{2}\Big\|
\frac{\partial u_{h}^{+}}{\partial\overrightarrow{n}}\Big\|  _{2,\Gamma}%
^{2}+C_{\mu}\left\|  F_{2}\right\|  _{2,\Gamma}^{2}\\
\leq C\mu^{2}\left|  u^{+}\right|^2  _{E^{2+\alpha}(\overline{\Omega}_{T}^{+}%
)}+C_{\mu}\big(  \left|  F_{2}\right|  _{\Gamma_{T}}^{(0)}\big)
^{2},\label{A6.13}%
\end{equation}%
\begin{equation}
\Bigg|\,  \int\limits_{\Gamma}u_{h}^{-}F_{3}\,dS\Bigg|  \leq
C\mu^{2}\left| u^{-}\right|^2
_{E^{2+\alpha}(\overline{\Omega}_{T}^{-})}+C_{\mu}\big(
\left|  F_{3}\right|  _{\Gamma_{T}}^{(0)}\big)  ^{2}.\label{A6.14}%
\end{equation}
We note also that, at $T<1,$
\[
\left|  \widetilde{u}^{\pm}\right|  _{\Omega_{T}^{\pm}}^{(0)}\leq
CT^{\alpha }\left\langle \widetilde{u}^{\pm}\right\rangle
_{t,\Omega_{T}^{\pm}}^{(\alpha
)}\leq CT^{\alpha}\left|  u^{\pm}\right|  _{E^{2+\alpha}(\overline{\Omega}%
_{T}^{\pm})},
\]%

\[
\left|  \widetilde{\rho}\right|  _{\Gamma_{T}}^{(0)}+\left|
\rho_{h}\right| _{\Gamma_{T}}^{(0)}\leq CT^{1-\alpha}\Big|
\frac{\partial\rho}{\partial t}\Big|  _{\Gamma_{T}}^{(0)}\leq
CT^{1-\alpha}\left|  \rho\right| _{P^{2+\alpha}(\Gamma_{T})},
\]%

\[
\left|  \nabla\rho_{h}\right|  _{\Gamma_{T}}^{(0)}\leq C\left\langle
\nabla\rho\right\rangle _{t,\Gamma_{T}}^{(\alpha)}\leq CT^{\frac{1}{2}%
}\left\langle \nabla\rho\right\rangle _{t,\Gamma_{T}}^{(\frac{1+\alpha}{2}%
)}\leq CT^{\frac{1}{2}}\left|  \rho\right|
_{P^{2+\alpha}(\Gamma_{T})}.
\]
Thus, the functions $F_{1}^{\pm},$ $F_{2},$ and $F_{3}$ in
\eqref{A6.4}--\eqref{A6.8} satisfy the estimate
\begin{multline}
\left|  F_{1}^{+}\right|  _{\overline{\Omega}_{T}^{+}}^{(0)}+\left|  F_{1}%
^{-}\right|  _{\overline{\Omega}_{T}^{-}}^{(0)}+\left| F_{2}\right|
_{\Gamma_{T}}^{(0)}+\left|  F_{3}\right| _{\Gamma_{T}}^{(0)}\\\leq
CT^{\lambda}\big(  \left|  u^{+}\right|  _{E^{2+\alpha}(\overline{\Omega}%
_{T}^{+})}+\left|  u^{-}\right|  _{E^{2+\alpha}(\overline{\Omega}_{T}^{-}%
)}+\left|  \rho\right|  _{P^{2+\alpha}(\Gamma_{T})}\big)  +C\mathcal{M}_{T}\label{A6.15}%
\end{multline}
with some $\lambda >0.$

We note that, by virtue of conditions \eqref{A6.7}, the inequality
\[
\int\limits_{\Omega^{\pm}}\left(  u_{h}^{\pm}\right)  ^{2}dx\leq
C\int \limits_{\Omega^{\pm}}\left(  \nabla u_{h}^{\pm}\right) ^{2}dx
\]
is valid. Then relations \eqref{A6.11}--\eqref{A6.15} yield
\begin{equation}
\max_{t\in\lbrack0,T]}\left\|  u_{h}^{\pm}\right\|
_{2,\Omega^{\pm}} \\
\leq C(\mu+C_{\mu}T^{\lambda})\big(\left|  u^{+}\right|  _{E^{2+\alpha}%
(\overline{\Omega}_{T}^{+})}+\left|  u^{-}\right|  _{E^{2+\alpha}%
(\overline{\Omega}_{T}^{-})}+\left|  \rho\right|  _{P^{2+\alpha}(\Gamma_{T}%
)}\big)\\+C_{\mu}\mathcal{M}_{T}.\label{A6.16}%
\end{equation}
By choosing firstly $\mu$ and then $T$ to be sufficiently small and
by joining estimates \eqref{A6.16}, \eqref{A6.3}, and \eqref{A6.1},
we obtain that estimate \eqref{1.99} is satisfied on some interval
$[0,T]$ independent of the values of the right-hand sides of the
problem.

By moving now upward along the axis $t$ step-by-step, as it was made
in \cite[Chap.~IV]{D10}, we prove estimate \eqref{1.99} on any
finite time interval $[0,T].$

Thus, Lemma \ref{L4.1} and estimate \eqref{1.99} are proved.
\end{proof}

%--------------------------------

Let us continue the proof of Theorem \ref{T3}. We write problem
\eqref{1.91}--\eqref{1.95} in the form
\begin{equation}
-\bigtriangleup
u^{\pm}+b^{\pm}(x,t)u^{\pm}=f_{1}^{\pm},\quad(x,t)\in\Omega
_{T}^{\pm},\label{A6.17}%
\end{equation}%
\begin{equation}
u^{+}-u^{-}=-A(x,t)\rho+f_{2},\quad(x,t)\in\Gamma_{T},\label{A6.18}%
\end{equation}%
\begin{equation}
a^{+}\frac{\partial
u^{+}}{\partial\overrightarrow{n}}-a^{-}\frac{\partial
u^{-}}{\partial\overrightarrow{n}}=f_{3}^{+}-f_{3}^{-}+\overrightarrow
{H}\nabla\rho,\quad(x,t)\in\Gamma_{T},\label{A6.19}%
\end{equation}%
\begin{equation}
u^{\pm}=f_{4}^{\pm},\quad(x,t)\in\Gamma_{T}^{\pm},\label{A6.20}%
\end{equation}%
\begin{equation}
u^{\pm}(x,0)=0,\quad x\in\overline{\Omega}^{\pm},\label{A6.21}%
\end{equation}%
\begin{equation}
\rho_{t}-\varepsilon\bigtriangleup_{\Gamma}\rho=f_{3}^{+}-a^{+}\frac{\partial
u^{+}}{\partial\overrightarrow{n}}-\sum_{i}h_{i}^{+}\rho_{\omega_{i}}%
,\quad\rho(\omega,0)=0.\label{A6.22}%
\end{equation}

It follows from results in \cite{LRU,SolEl} that, for the given
function $\rho\in E^{2+\alpha}(\Gamma_{T})$ on the right-hand sides
of relations \eqref{A6.17}--\eqref{A6.20}, the problem of
conjugation \eqref{A6.17}--\eqref{A6.20} has the unique solution
that satisfies the estimate
\begin{equation}
\left|  u^{+}\right|
_{E^{2+\alpha}(\overline{\Omega}_{T}^{+})}+\left| u^{-}\right|
_{E^{2+\alpha}(\overline{\Omega}_{T}^{-})}\leq
C\mathcal{M}_{T}+C\left|
\rho\right|  _{E^{2+\alpha}(\Gamma_{T})},\label{A6.23}%
\end{equation}
so that
\begin{equation}
\left|  \nabla u^{+}\right|
_{E^{1+\alpha}(\overline{\Omega}_{T}^{+})}\leq
C\mathcal{M}_{T}+C\left|  \rho\right|  _{E^{2+\alpha}(\Gamma_{T})}.\label{A6.24}%
\end{equation}
We note that the results in \cite{LRU,SolEl} concern the spaces
$H^{2+\alpha}(\overline{\Omega}),$ but the transition to the spaces
$E^{2+\alpha}(\overline{\Omega}_{T})$ is realized simply by the
consideration of the appropriate problem for the functions
$u_{h}^{\pm}=( u_{h}^{\pm}(x,t) -u_{h}^{\pm}(x,t-h)) /h^{\alpha}.$

Thus, we have correctly defined the operator $L_{\varepsilon}: \rho
\rightarrow L_{\varepsilon}\rho$ that puts each function $\rho\in
E^{2+\alpha}(\Gamma_{T})$ given on the right-hand sides of relations
\eqref{A6.17}--\eqref{A6.20} and \eqref{A6.22} in correspondence
with the function $L_{\varepsilon}\rho$ that is the solution of the
Cauchy problem \eqref{A6.22} with the function $\rho$ and the
function $\partial u^{+}/\partial \overrightarrow{n}$ determined by
$\rho$ that are given on the right-hand side of \eqref{A6.22}.

Relations \eqref{A6.23} and \eqref{A6.24} and results in \cite{M}
yield
\begin{multline}
\left|  L_{\varepsilon}\rho\right| _{P^{3+\alpha}(\Gamma_{T})}\leq
C_{\varepsilon}\big(  \left|  \nabla u^{+}\right|  _{E^{1+\alpha}%
(\overline{\Omega}_{T}^{+})}+\left|  \nabla\rho\right|  _{E^{1+\alpha}%
(\Gamma_{T})}+\left|  f_{3}^{+}\right|
_{E^{1+\alpha}(\Gamma_{T})}\big) \\
\leq C_{\varepsilon}\big(  \left|  f_{3}^{+}\right|  _{E^{1+\alpha}%
(\Gamma_{T})}+\left|  \rho\right| _{E^{2+\alpha}(\Gamma_{T})}\big)
,\label{A6.25}%
\end{multline}
and, in addition, for $\rho_{1,}\rho_{2}\in E^{2+\alpha}(\Gamma_{T}),$%
\begin{equation}
\left|  L_{\varepsilon}\rho_{2}-L_{\varepsilon}\rho_{1}\right|
_{P^{3+\alpha }(\Gamma_{T})}\leq C_{\varepsilon}\left|
\rho_{2}-\rho_{1}\right|
_{E^{2+\alpha}(\Gamma_{T})}.\label{A6.26}%
\end{equation}

It follows from results in \cite{Lun,SolIn1,SolIn2} about the
interpolation in H\"{o}lder spaces that the quantity
\begin{equation}
\left[  D_{x}^{2}\rho\right]
_{x,t,\Gamma_{T}}^{(\alpha,\frac{1}{2}+\alpha
)}\leq C\left|  \rho\right|  _{P^{3+\alpha}(\Gamma_{T})}\label{A6.27}%
\end{equation}
is finite. Hence, since $\rho(x,0)=0,$ $\rho_{t}(x,0)=0,$ we have
\begin{equation}
\left[  D_{x}^{2}\rho\right] _{x,t,\Gamma_{T}}^{(\alpha,\alpha)}\leq
CT^{\frac{1}{2}}\left[ D_{x}^{2}\rho\right]
_{x,t,\Gamma_{T}}^{(\alpha ,\frac{1}{2}+\alpha)}\leq
CT^{\frac{1}{2}}\left|  \rho\right| _{P^{3+\alpha
}(\Gamma_{T})}.\label{A6.28}%
\end{equation}
Analogous inequalities with the factor $T^{\lambda}$ are valid also
for other terms in the definition of the norm $\rho$ in the space
$E^{2+\alpha}(\Gamma_{T}).$ For $\rho_{1,}\rho_{2}\in
E^{2+\alpha}(\Gamma_{T}),$ this result and relation \eqref{A6.26}
yield
\begin{equation}
\left|  L_{\varepsilon}\rho_{2}-L_{\varepsilon}\rho_{1}\right|
_{E^{2+\alpha }(\Gamma_{T})}\leq C_{\varepsilon}T^{\lambda}\left|
\rho_{2}-\rho_{1}\right|
_{E^{2+\alpha}(\Gamma_{T})}.\label{A6.29}%
\end{equation}
Thus, by choosing $T$ to be sufficiently small, we obtain that the
operator $L_{\varepsilon}$ is a contraction one on
$E^{2+\alpha}(\Gamma_{T})$ and, hence, has the single fixed point.
Together with \eqref{A6.25}, this yields a solution of problem
\eqref{1.91}--\eqref{1.95} on some interval $[0,T]$ independent of
the values of the right-hand sides of the problem. By moving upward
along the axis $t,$ as it was made in \cite[Chap.~IV]{D10}, we
obtain the solution of problem \eqref{1.91}--\eqref{1.95} from the
required class an any finite time interval. The estimate of the
solution was proved above in Lemma \ref{L4.1}.

Thus, we prove the assertion of the theorem for $\varepsilon>0.$ We
now transit to the limit as $\varepsilon \rightarrow 0.$ We note
that, by virtue of estimate \eqref{1.99}, the sequence
$u^{\pm}_{\varepsilon},$ $\rho_{\varepsilon}$ is compact in the
spaces $E^{2+\beta}(\overline{\Omega}^{\pm}_{T})$ and
$P^{2+\beta}(\Gamma_{T})$ with any $\beta<\alpha.$ Hence, we can
separate a subsequence $u^{\pm}_{\varepsilon_{n}}\rightarrow
u^{\pm},$ $\rho_{\varepsilon_{n}}\rightarrow \rho$ that converges in
these spaces, and the functions $u^{\pm}$ and $\rho$ present the
solution of problem \eqref{1.91}--\eqref{1.95} at $\varepsilon=0.$
Indeed, in view of the available estimate, the limit transition is
possible in each of the relations. In addition, the limiting
functions will belong to the same spaces
$E^{2+\alpha}(\overline{\Omega}^{\pm}_{T})$ and
$P^{2+\alpha}(\Gamma_{T}),$ because, by virtue of the estimate
uniform, for example, in $\varepsilon,$
\[
\Big|  \frac{D_{x}^{2}u_{\varepsilon_{n}}^{+}(x+\overrightarrow{l}%
,t)-D_{x}^{2}u_{\varepsilon_{n}}^{+}(x,t)}{\big|
\overrightarrow{l}\big|
^{\alpha}}\Big|  \leq C\mathcal{M}_{T},%
\]
we can transit to the limit in this inequality as
$\varepsilon_{n}\rightarrow 0$ due to the uniform convergence of the
functions $D^{2}_{x}u^{+}$ on $\overline{\Omega}^{+}_{T}.$ This
yields
\[
\left\langle D_{x}^{2}u_{\varepsilon_{n}}^{+}(x,t)\right\rangle
_{x,\overline {\Omega}_{T}^{+}}^{(\alpha)}\leq CM_{T}.
\]
The remaining estimates are analogous.

Eventually, the uniqueness of the solution that was obtained by the
limit transition follows directly from estimate \eqref{1.99}.

Thus, Theorem \ref{T3} is proved.
\end{proof}
%------------------------------------------------------------

\section{A nonlinear problem: the proof of Theorem \ref{T1}} \label{s5}

The proof of Theorem \ref{T1} is based on Theorem \ref{T3} and a
representation of the problem under consideration in the form
\eqref{1.31}--\eqref{1.35}. We now define a nonlinear operator
$\mathcal{F}(\psi),$ $\psi=(v^{+},v^{-},\delta)$ in
\eqref{1.31}--\eqref{1.35} that puts every given $\psi$ on the
nonlinear right-hand sides of relations \eqref{1.31}--\eqref{1.35}
in correspondence with the solution of the linear problem determined
by the left-hand sides of these relations. In this case, Theorem
\ref{T3} and Lemma~\ref{L1} imply that the operator
$\mathcal{F}(\psi)$ possesses the following properties on a ball
$\mathcal{B}_{r}=\{\psi: \|\psi\|\leq r\}\subset \mathcal{H}$:
\begin{equation}
\|\mathcal{F}(\psi)\|_{\mathcal{H}}\leq
C(T^{\alpha/2}+r)\|\psi\|_{\mathcal{H}}, \label{5.1}
\end{equation}
\begin{equation}
\|\mathcal{F}(\psi_{1})-\mathcal{F}(\psi_{2})\|_{\mathcal{H}}\leq
C(T^{\alpha/2}+r)\|\psi_{1}-\psi_{2}\|_{\mathcal{H}}. \label{5.2}
\end{equation}
It is easy to see that relations \eqref{5.1} and \eqref{5.2} imply
that, at sufficiently small $T$ and $r,$ the operator
$\mathcal{F}(\psi)$ maps the closed ball $\mathcal{B}_{r}$ into
itself and is a contraction operator there. The single fixed point
of this operator gives the solution of the initial nonlinear problem
with free boundary that is related to Theorem \ref{T1}. Thus,
Theorem \ref{T1} is proved.

%===================================================
%===================================================
%===================================================

%---------------------------------------

\end{document}